\documentclass[ 12pt]{amsart}
\usepackage{amsfonts}
\usepackage{mathrsfs}
\usepackage{bbm}
\usepackage{amsfonts,amsthm}
\usepackage{amssymb,amsmath,graphicx}
\usepackage{cases}
\usepackage{float}
\usepackage[colorlinks=true]{hyperref}
\hypersetup{urlcolor=blue, citecolor=red}
\usepackage{hyperref}
\usepackage{enumerate}
\usepackage{cite}
\newtheorem{theorem}{Theorem}

\newtheorem{lemma}{Lemma}

\newtheorem{remark}{Remark}

\def\Xint#1{\mathchoice
	{\XXint\displaystyle\textstyle{#1}}
	{\XXint\textstyle\scriptstyle{#1}}
	{\XXint\scriptstyle\scriptscriptstyle{#1}}
	{\XXint\scriptscriptstyle\scriptscriptstyle{#1}}
	\!\int}
\def\XXint#1#2#3{{\setbox0=\hbox{$#1{#2#3}{\int}$ }
		\vcenter{\hbox{$#2#3$ }}\kern-.6\wd0}}

\def\dashint{\Xint-}

\def\Div{{\rm div}}
\begin{document}
	\title[Low Mach number limit and Convergence Rate]{Low Mach Number Limit and Far Field Convergence Rates of Potential Flows in Multi-Dimensional Nozzles With an Obstacle Inside}
\author{Lei Ma}
\address{School of Mathematical Sciences, Shanghai Jiao Tong University, 800 Dongchuan Road, Shanghai, 200240, China}
\email{yutianml@sjtu.edu.cn}
\author{Tian-Yi Wang}
\address{Department of Mathematics, School of Science, Wuhan University of Technology, 122 Luoshi Road, Wuhan, Hubei, 430070, China}
\email{tianyiwang@whut.edu.cn; wangtianyi@amss.ac.cn}
	\author{Chunjing Xie}
\address{School of mathematical Sciences, Institute of Natural Sciences,
Ministry of Education Key Laboratory of Scientific and Engineering Computing,
and SHL-MAC, Shanghai Jiao Tong University, 800 Dongchuan Road, Shanghai, China}
\email{cjxie@sjtu.edu.cn}

	\date{}
	\begin{abstract}
			This paper considers the low Mach number limit and far field convergence rates of steady Euler flows with external forces in three-dimensional infinitely long nozzles with an obstacle inside. First, the well-posedness theory for both incompressible and compressible subsonic flows with external forces in multidimensional nozzle with an obstacle inside are established by several uniform estimates. The uniformly subsonic compressible flows tend to the incompressible flows as quadratic order of Mach number as the compressibility parameter goes to zero. Furthermore, we also give the convergence rates of both incompressible flow and compressible flow at far fields as the boundary of nozzle goes to flat even when the forces do not admit convergence rate at far fields. The convergence rates obtained for the flows at far fields clearly describe the effects of the external force.
			\\[3mm]{\bf Keywords:}   Subsonic flows, Potential equation, Nozzles, Obstacle, Low Mach number limit, Convergence rates
	\end{abstract}
	\maketitle
\section{Introduction and main results}\label{introdunction}
Both incompressible and compressible Euler equations can be used to describe the motion of the fluid dynamics and they give rise to many significant problems in mathematical theory. One of the important topics is the low Mach number limit problem which considers the convergence of compressible flows to the incompressible ones as the compressibility parameter goes to zero.
The isentropic compressible Euler equations with the external force are described as follows
\begin{equation}\label{Eulereq}
\begin{cases}
\Div(\rho  {\bf u}) = 0,   \\
\Div(\rho {\bf u}\otimes {\bf u})+\nabla p=\rho {\bf F},
\end{cases}
\end{equation}
where $\rho$ represents density, ${\bf u}=(u_1,u_2,u_3)$ is the flow velocity, $p$ is the pressure which is a function of $\rho$, and ${\bf F}$ describe the external forces, respectively. In this paper, denote
\begin{equation}
p=\frac{\tilde p(\rho)-\tilde p(1)}{\epsilon^2},
\end{equation} where $\epsilon$ is the compressibility parameter\cite{MR2157144}. We always assume \begin{equation}\label{pressurecondtion}
\tilde p^{\prime}(\rho)>0\quad \text{and}\quad \rho\tilde  p^{\prime\prime}(\rho)+2\tilde p'(\rho)\geq0, \quad \text{for } \rho>0.
\end{equation} The typical example is the polytropic gas for which the pressure is given by $\tilde p=\rho^{\gamma}$ with $\gamma>1$  called the adiabatic exponent.

Assume the flow is irrotational, $i.e.,$
\begin{equation}\label{irrotational}
curl\  {\bf u}=0,
\end{equation}
then, in the simply connected domain, there exists a potential function $\varphi$ such that
\begin{equation}
\nabla\varphi={\bf u}.
\end{equation}
Due to the irrotational condition \eqref{irrotational}, one may assume that the external force is conservative, which means that there exists a function $\phi$ such that
\begin{equation*}{\bf F}=\nabla\phi.\end{equation*}

Denote \begin{equation}c(\rho)=\sqrt{p^{\prime}(\rho)}=\frac{\sqrt{\tilde p'(\rho)}}{\epsilon}, \end{equation} which is called the sound speed. And the Mach number is defined as
\begin{equation}
M=\frac{|{\bf u}|}{c}=\frac{\epsilon|{\bf u}|}{\sqrt{\tilde p'(\rho)}}.
\end{equation}
The flow is said to be subsonic, sonic, supersonic when $M<1$, $=1$, $>1$, respectively. Therefore, if $\frac{|{\bf u}|}{\sqrt{\tilde p'(\rho)}}$ is bounded and has a lower positive bound, one has $M\sim O(\epsilon)$.

We consider the domain to be a nozzle $\tilde\Omega$ which contains an obstacle $\Omega'$ inside, which models the wind tunnel in reality.  Moreover,  assume  that $\partial\tilde\Omega$ and $\partial\Omega'$ are $C^{2,\alpha}$. Let $B_1(0)\subset \mathbb{R}^2$ be the unit disk centered at origin.  Let
\begin{equation}\label{cylinder}
\mathcal C=B_1(0)\times(-\infty,+\infty)\end{equation} be the perfect cylinder.
Suppose that there exists an invertible $C^{2,\alpha}$ map $\mathcal R:\tilde\Omega\rightarrow \mathcal C $ satisfying
\begin{equation}\label{tildeOmega}
\begin{cases}
\mathcal R(\partial \tilde\Omega)=\partial\mathcal C,\\
\mathcal R(\tilde\Omega\cap\{x_3=k\})=B_1(0)\times\{y_3=k\}\quad \text{for any }k\in \mathbb R,\\
\|\mathcal R\|_{C^{2,\alpha}}+	\|\mathcal R^{-1}\|_{C^{2,\alpha}}\leq C,
\end{cases}
\end{equation}
where $C$ is a constant.
Furthermore,  assume that there exists a $C^{2,\alpha}$ map $\mathcal Y$ such that
\begin{equation}\label{Omega'}
\mathcal Y(\Omega')\rightarrow \mathcal B,
\end{equation} where $\mathcal B$ is the three dimensional unit ball with center at the origin. Using the cylindrical coordinates, $\tilde\Omega$ and $\Omega'$ can be written as
\begin{equation}\label{Omega}
\tilde\Omega=\bigg\{(r,\tau,x_3)\big|\,x_1=r\cos\tau,\, x_2=r\sin\tau,\, r< f_1(\tau,x_3),\,\tau\in[0,2\pi),\,x_3\in\mathbb{R}\bigg\}
\end{equation} and
\begin{equation}\label{Omega_prime} \Omega^\prime=\bigg\{(r,\tau,x_3)\big|\,x_1=r\cos\tau,\, x_2=r\sin\tau, r< f_2(\tau,x_3),\,\tau\in[0,2\pi),\,L_1\leq x_3\leq L_2\bigg\},
\end{equation}
respectively, where $L_1$ and $L_2$ are constants.
Assume
\begin{equation}
0\leq f_2\leq C\quad \text{and}\quad\frac{1}{C}\leq f_1-f_2\leq C\quad\text{for any}\, \tau\in [0,2\pi),\,x_3\in [L_1,L_2]
\end{equation}
and
\begin{equation}
\frac{1}{C}\leq f_1\leq C\quad\text{for any}\, \tau\in [0,2\pi),\,x_3\in\mathbb{R},
\end{equation}
where $C$ is a positive constant. Without loss of generality, assume the origin $O\in\Omega'$. In the rest of the paper, denote
\begin{equation}
{\Omega}=\tilde\Omega\setminus\Omega'\quad\text{and}\quad	\Sigma_t=\Omega\cap\{x_3=t\}.
\end{equation}
The slip boundary condition is supplemented on $\partial \Omega$, i.e.,
\begin{equation}\label{slipboundary}
{\bf u}\cdot{\bf n}=0 \quad \text{on}\,\, \partial \Omega,
\end{equation}
where $\textbf{n}=(n_1,n_2,n_3)$ is unit outer normal of $\partial \Omega$.
On any cross section $S_0$ of $\Omega$, it follows from the continuity equation in \eqref{Eulereq} that
\begin{equation}\label{massflux}
\int_{S_0}\rho {\bf u}\cdot {\bf{l}}ds=m_0,
\end{equation} where $\textbf{l}$ is the unit normal pointed to the right of $S_0$. $m_0$ is called the mass flux of the flow across the nozzle, which is conserved through each cross section.

The well-posedness problem on compressible subsonic Euler flows in infinitely long nozzles was posed in \cite{MR0096477}. The first rigorous proof for the well-posedness of irrotational flows in two dimensional nozzles was established in \cite{MR2375709} via the stream function formulation. The results were extended to the three dimensional axially symmetric case in \cite{MR2644144}. The uniformly subsonic flows in general multidimensional nozzles were obtained via potential formulation and variational method in \cite{MR2824469}. Afterwards, the results were extended to the subsonic flows with external force in \cite{MR3638912}. Recently, the optimal convergence rates of velocity at far field was given in \cite{MaLei}. When the vorticity of subsonic flow is not zero, the existence of solutions was established in \cite{MR2607929} as long as the mass flux is less than a critical value and the variation of the Bernoulli's function is small. Later on, the well-posedness of subsonic flows with large vorticity is proved in \cite{MR3196988} under the condition that the velocity at upstream is convex. It is worth pointing that the optimal convergence rates of the flows at far fields are also established in \cite{MR3196988}. When the conditions of the smallness of variation of Bernoulli's function or the convexity  of velocity are removed, the existence of general subsonic flows with characteristic discontinuity was obtained in \cite{MR3914482}. There are many literatures on the subsonic flows in infinitely long nozzles, see\cite{MR2879733,MR2737815,MR3537905,MR3861290} and reference therein.

For the exterior problems, which is to study the flow past around the given obstacles, the existence of two dimensional subsonic flows was first studied in \cite{Frankl1934Keldysh}. Later on, the existence of subsonic solution around a smooth body was established in \cite{shiffman1952existence} when the circulation is prescribed and the free stream Mach number is less than a critical value. Bers proved the existence of subsonic flows even when the body has a sharp corner \cite{Bers1954Existence}. In \cite{CPA:CPA3160100102}, the uniqueness and the convergence rates of subsonic plane flows were given. The existence of subsonic flows past a three dimensional body was first investigated in \cite{Finn1957Three} and later established in \cite{Dong1977Nonlinear,Dong1993Subsonic} when the Mach number of the free stream flow is less than a critical value. When there is an external force, the existence of subsonic flows past a body was established in \cite{GuandWang}.

A physical interesting problem is to study the low Mach number limit of the compressible subsonic flows \cite{MR2549370,MR2157144}. More precisely, as the compressibility parameter tends to zero, whether the solutions of compressible Euler equations \eqref {Eulereq} converge to that of the homogeneous incompressible Euler equations
\begin{equation}
\begin{cases}
\Div {\bf u}=0,\\
\Div({\bf u}\otimes{\bf u})+\nabla p={ \bf F},
\end{cases}
\end{equation}
where ${\bf u}$ and $p$ represent the velocity and pressure, respectively.

The first mathematical theory of the low Mach number limit for steady irrotational flows was studied in  \cite[Sect. 47]{MR0119651} and \cite{MR0176702} where  the solutions were written as power series of Mach number. Klainerman and Majda \cite{MR615627,MR668409} proceeded with the study of the convergence of unsteady compressible flows to  incompressible case for the suitable data. Later on, Ukai\cite{MR849223} improved the low Mach number limits for the general data with help of the decay property of acoustic waves. An important progress is made by M$\acute{e}$tivier and Schochet\cite{MR1834114} where they proved the low Mach number limit in the whole space for general initial data of full Euler systems.  These results were extended to the exterior problems in \cite{MR2106119}. The low Mach number limit for one dimensional problem was investigated in the $BV$ space \cite{MR2403601}. The first rigorous analysis of subsonic flows for the steady Euler equations past a body, in infinitely long nozzles, and largely open nozzle were obtained in \cite{2019arXiv190104320L}, \cite{2019arXiv190101048L}, and \cite{Zhang-Wang}, respectively.

The aim of this paper is to investigate the well-posedness of subsonic flows with force through infinitely long nozzles with an obstacle inside and the low Mach number limit of the associated flows. The next key issue is to study the convergence rates of  the velocity at far fields.

In order to study the low Mach number limit, we first investigate  the incompressible flows in $\Omega$.
 Find $\bar{\bf u}$ and $\bar p$ satisfy
\begin{equation}\label{problemI1}
\begin{cases}
\Div {\bf \bar u}=0, &\text{ in $\Omega$,}\\
\Div({\bf \bar u}\otimes{\bf \bar u})+\nabla \bar p=\nabla\phi, &\text{ in $\Omega$,}\\
{\bf \bar u}\cdot{\bf n}=0, &\text{ on } \partial\Omega,\\
\int_{\Sigma_t} \bar{\bf u}\cdot {\bf{l}}ds=m_0,
\end{cases}
\end{equation}
where $\textbf{n}$ is the unit outer normal of $\Omega$ and  $\textbf{l}$ is the unit normal pointed to the right of $\Sigma_t$.
\begin{theorem}\label{incompressible case}
	For any $m_0>0$, suppose \begin{equation}\label{consevative force}
	\phi\in L^\infty(\Omega)\text{ and } \nabla\phi\in L^q(\Omega) \text{ for } q>3,\end{equation} there exists a solution $({\bf \bar u},\bar p)\in\big(C^{\alpha}(\Omega) \big)^4$ to problem \eqref{problemI1}. Moreover, let $K$ be a large positive number and $q_0$  be the constant satisfying $q_0|B_1(0)|=m_0$. \\ (i) If the nozzle is flat at the downstream, i.e.,
	\begin{equation}\label{Icylindercase} \Omega\cap\{x_3>K\}=B_1(0)\times(K,+\infty), \end{equation} there exists a positive constant $\mathfrak d_1$ such that
	\begin{equation}\label{Icylinder boundary}
	|{\bf \bar u}-(0,0, q_0)|\leq Ce^{-\mathfrak d_1 x_3}, \quad\text{for $x_3>K$};
	\end{equation}\\
	(ii) If the boundary of the nozzle satisfies
	\begin{equation}\label{Iboundary_algebratic_rateii}
	\sum\limits_{k=0}^{2}\big|x_3^k\partial_{3}^k(f_1-1)\big|\leq C{x_3^\mathfrak {-a_1}},\quad\text{for $x_3>K$},
	\end{equation}
	with $\mathfrak a_1>0$, then the velocity satisfy
	\begin{equation}\label{Ibdecay}
	|{\bf \bar u}-(0,0, q_0)|\leq Cx_3^{-\mathfrak a_1},
	\end{equation}
	where $C$ is a constant independent of  $x_3$.\end{theorem}
Because of the irrotational condition \eqref{irrotational}, one has
\begin{equation}
\nabla\bar\varphi={\bf\bar u}.
\end{equation}
The straightforward computations yield the Bernoulli's law, $i.e.$,
\begin{equation}
\bar p=\phi-\frac{|\nabla\bar\varphi|^2}{2}+C,
\end{equation} where $C$ is a constant. Thus, the problem \eqref{problemI1} is converted to the
following problem
\begin{equation}\label{problemI2}
\begin{cases}
\triangle\bar\varphi=0,&\text{ in }\Omega,\\
\frac{\partial\bar\varphi}{\partial{\bf n}}=0, &\text{ on } \partial\Omega,\\
\int_{\Sigma_t} \nabla\bar\varphi\cdot {\bf{l}}ds=m_0.
\end{cases}
\end{equation}

Now, we turn to the compressible case. Suppose that $(\rho^\epsilon, {\bf u}^\epsilon)$ satisfies
\begin{equation}\label{Eulereq1}
\begin{cases}
\Div(\rho^\epsilon  {\bf u^\epsilon}) = 0, &\text{ in $\Omega$},  \\
\Div(\rho^\epsilon {\bf u^\epsilon}\otimes {\bf u^\epsilon})+\nabla p^\epsilon=\rho^\epsilon\nabla\phi,&\text{ in $\Omega$},\\
{\bf  u^\epsilon}\cdot{\bf n}=0,  &\text{ on } \partial\Omega,\\
\int_{\Sigma_t}\rho^\epsilon {\bf u^\epsilon}\cdot {\bf{l}}ds=m_0,
\end{cases}
\end{equation}
with $p^\epsilon$ satisfies
\begin{equation}\label{p_epsilon}
p^\epsilon=\frac{\tilde p(\rho)-\tilde p(1)}{\epsilon^2}.
\end{equation}
\begin{theorem}\label{Theorem low Mach limit}For compressible flows, suppose \eqref{consevative force} holds, for any $m_0>0$, there exists a constants $\epsilon_c$ such that if $0<\epsilon<\epsilon_c$ the problem \eqref{Eulereq1} admits a unique solution $({\bf u^\epsilon},\rho^\epsilon,p^\epsilon)\in \big(C^{\alpha}(\Omega) \big)^5$ for some $\alpha<1$ with $M<1$. Furthermore, as $\epsilon\rightarrow0$ one has
	\begin{equation}
	\rho^\epsilon=1+O(\epsilon^2),\quad {\bf u^{\epsilon}}={\bf\bar u}+O(\epsilon^2),\quad  p^\epsilon=\bar p+O(\epsilon^2) \quad\text{and}\quad M=O{(\epsilon)},
	\end{equation}
	where $(\bar{\bf{u}},\bar p)$ solves the problem \eqref{problemI1} in Theorem \ref{incompressible case}.
	\end{theorem}

Based on the existence of subsonic solution to the problem \eqref{Eulereq1}, if the boundary of nozzle tends to be flat at far fields, we can also obtain the convergence rates of velocity fields. Let $\mathcal C$ be the perfect cylinder defined in \eqref{cylinder}.
According to \cite{MR3638912}, for any $m_0>0$,
there exists a $\hat\epsilon_c>0$ such that for any $\epsilon<\hat\epsilon_c$, as long as the force $\phi$ satisfies \eqref{consevative force},  there exists a unique uniformly subsonic solution $\bf{u}_{*}$ satisfying
\begin{equation}\label{Eulereq2}
\begin{cases}
\Div(\rho_{*}  {\bf u_{*}}) = 0, &\text{ in $\mathcal C$},  \\
\Div(\rho_{*} {\bf u_{*}}\otimes {\bf u_{*}})+\nabla p_{*}=\rho_*\nabla\phi,&\text{ in $\mathcal C$},\\
{\bf  u_{*}}\cdot{\bf n}=0,  &\text{ on } \partial\mathcal C,\\
\int_{B_1(0)}\rho_{*} {\bf u_{*}}\cdot {\bf{l}}ds=m_0.
\end{cases}
\end{equation}
It is worth pointing that if $\phi$ is a function independent of $x_3$, $i.e$., $\phi=\bar\phi(x_1,x_2)$ and satisfies  \eqref{consevative force}, the straightforward computations yield that  ${
	\bf {\bar u_*}}= (0,0,\bar q)$ is the corresponding solution to $\eqref{Eulereq2}$,
where
$\bar q$ is a constant satisfying
\begin{equation}
\int_{B_1(0)}\rho^\epsilon(\bar q^2,\bar\phi)\bar qdx'=m_0\quad \text{and}\quad \bar q<\frac{\sqrt{\tilde p'(\rho^\epsilon)}}{\epsilon}.
\end{equation}
\begin{theorem}\label{Theorem convergence rate}
	Let $K$ be a large positive number. For any fixed $0<\epsilon<\min\{\epsilon_c,\hat\epsilon_c\}$, ${\bf u}^\epsilon$ is the subsonic solution of \eqref{Eulereq1} and ${\bf u}_{*}$ is the solution of \eqref{Eulereq2}. \\
	(i) Suppose that the nozzle is flat, $i.e.$,
	$\Omega$ satisfies \eqref{Icylindercase}.  If $\phi$ satisfies \eqref{consevative force}, there exists a positive constant $\mathfrak d$ such that
	\begin{equation}\label{cylinder boundary}
	|{\bf u}^\epsilon-{\bf u}_{*}|\leq Ce^{-\mathfrak d x_3} \quad\text{for $x_3>K$},
	\end{equation}
	where $C$ is a constant independent of $x_3$.\\
	(ii) If the boundary satisfies \eqref{Iboundary_algebratic_rateii}
	with $\mathfrak a_1>1$, then the velocity satisfies
	\begin{equation}
	|{\bf u}^\epsilon-{\bf u}_{*}|\leq Cx_3^{-\mathfrak a_1+1},
	\end{equation}
	where $C$ is a constant independent of $x_3$.\\
	(iii)For given $\mathfrak a_1>0$ and $\mathfrak b_1>0$, suppose that the boundary of the nozzle satisfies \eqref{Iboundary_algebratic_rateii}.
	if, in addition, the conservative force $\phi$ satisfies
	\begin{equation}\label{force algebratic rate}
	\sum\limits_{k=0}^{2}\big|x_3^k\partial_{3}^k(\phi-\bar\phi)\big|\leq{C}{x_3^\mathfrak {-b_1}}\quad\text{for $x_3>K$},
	\end{equation}
	then the velocity satisfies
	\begin{equation}\label{Thm velocity}
	|{\bf u}^\epsilon-{\bf \bar u}_{*}|\leq Cx_3^{-\mathfrak b},
	\end{equation}
	where $\mathfrak b=\min\{\mathfrak a_1,\mathfrak b_1\}$ and $C$ is a constant independent of $x_3$.
\end{theorem}
As long as the flows are irrotational, there also exists a potential $\varphi^{\epsilon}$ such that
\begin{equation}\label{u_epsilon}
\nabla\varphi^{\epsilon}={\bf u}^\epsilon.
\end{equation}
Similarly, one has the following Bernoulli's law
\begin{equation}\label{bernoulislaw}
\frac{|\nabla\varphi^\epsilon|^2}{2}+\int_{1}^{\rho^\epsilon}\frac{(p^\epsilon(s))'}{s}ds=\phi+C_1,
\end{equation}
where $C_1$ is a constant.
Let $\tilde h$ be the enthalpy function satisfying
\begin{equation}
\tilde h'(\rho)=\frac{\tilde p'(\rho)}{\rho}.
\end{equation}
Without loss of generality, assuming $C_1=0$, then  \eqref{bernoulislaw} becomes
\begin{equation}\label{bernouli}
\frac{|\nabla\varphi^\epsilon|^2}{2}+h^\epsilon(\rho^\epsilon)-h^\epsilon(1)=\phi,
\end{equation}
where $h^\epsilon(\rho)=\epsilon^{-2}\tilde{h}(\rho)$.
The straightforward computations yield that $\tilde h(\rho)$ is a strictly increasing function with respect to $\rho$, so is $h^\epsilon(\rho)$.

Now, one may introduce the critical speed for the flows. For each fixed $0<\theta\leq1$, one can follow \cite{2019arXiv190101048L} to define $q_{\theta}^\epsilon$ such that ${\bf u^\epsilon}<q_{\theta}^{\epsilon}$ if and only if $M<\theta$. Specially, for the polytrotic case, define
\begin{equation}
\mu^2=\frac{(\gamma-1)\theta^2}{2+(\gamma-1)\theta^2}\quad\text{and}\quad q_\theta^\epsilon=\mu\sqrt{2\big(\phi+h(1)\big)},
\end{equation}
then  the Bernoulli's function \eqref{bernouli} can be written in the following form
\begin{equation}
\mu^2|\nabla\varphi^\epsilon|^2-(1-\mu^2)\theta^2c^2=(q_\theta^\epsilon)^2.
\end{equation}
This implies that
\begin{equation}
|\nabla\varphi^\epsilon|^2-(q_\theta^\epsilon)^2=(1-\mu^2)(|\nabla\varphi^\epsilon|^2-\theta^2c^2).
\end{equation}
In particular, when $\theta=1$,  $q_{cr}^\epsilon:=q_1^\epsilon$ is called the critical speed.
Obviously, the critical speed $q_{cr}^\epsilon\sim O(\epsilon^{-1})$.
It easy to see that  $|{\bf{u}^\epsilon}|<q_{cr}^\epsilon(\phi)$ holds if and only if the flow is subsonic, $i.e.$, $M<1$. Similarly, when pressure satisfies \eqref{pressurecondtion}, for each fixed positive $\theta<1$,  $q^{\epsilon}_\theta(\phi)$ satisfies $M<\theta$ holds if and only if $|{\bf u}^\epsilon|<q_{\theta}^{\epsilon}(\phi)$. And, $q_\theta^\epsilon(\phi)$ is increasing with respect to $\theta\in(0,1)$. In addition, $\epsilon q_\theta^\epsilon(\phi)$ and $\epsilon q_{cr}^\epsilon(\phi)$ are bounded with respect to $\epsilon$.

Because of \eqref{bernouli}, the densty $\rho^\epsilon$ can be represented as follows
\begin{equation}\label{rho_epsilon}
\rho^\epsilon(|\nabla\varphi^\epsilon|^2,\phi)=\tilde h^{-1}\bigg(\frac{\epsilon^2(2\phi-|\nabla\varphi^\epsilon|^2)}{2}+\tilde h(1)\bigg).
\end{equation}
Thus for each fixed $\epsilon$,  \eqref{Eulereq1} is equivalent to the following problem
\begin{equation}\label{problem C2}
\begin{cases}
\Div\big(\rho^\epsilon(|\nabla\varphi^\epsilon|^2,\phi)\nabla\varphi^\epsilon\big)=0,&\text{ in }\Omega,\\
\frac{\partial\varphi^\epsilon}{\partial{\bf n}}=0, &\text{ on } \partial\Omega,\\
\int_{\Sigma_t}\rho^\epsilon(|\nabla\varphi^\epsilon|^2,\phi) \nabla\varphi^\epsilon\cdot {\bf{l}}ds=m_0.
\end{cases}
\end{equation}

Furthermore, there exists a potential $\varphi_*$ such that the problem \eqref{Eulereq2} is equivalent to the following
\begin{equation}\label{cylinderproblem}
\begin{cases}
\Div\big(\rho^\epsilon(|\nabla\varphi_*|^2,\phi)\nabla\varphi_*\big)=0,&\text{ in }\mathcal C,\\
\frac{\partial\varphi_*}{\partial{\bf n}}=0, &\text{ on } \partial\mathcal C,\\
\int_{B_1(0)}\rho^\epsilon(|\nabla\varphi_*|^2,\phi) \nabla\varphi_*\cdot {\bf{l}}ds=m_0,&\text{ for } x_3\in\mathbb R.
\end{cases}
\end{equation}
Therefore, if $\phi=\bar\phi(x_1,x_2)$ and satisfies  \eqref{consevative force}, then the straightforward computations yield that $\bar\varphi_*= \bar qx_3$ satisfies
\begin{equation}\label{limitcylinderproblem}
\begin{cases}
\Div\big(\rho^\epsilon(|\nabla\bar\varphi_*|^2,\bar\phi)\nabla\bar\varphi_*\big)=0,&\text{ in }\mathcal C,\\
\frac{\partial\bar\varphi_*}{\partial{\bf n}}=0, &\text{ on } \partial\mathcal C\\
\int_{B_1(0)}\rho^\epsilon(|\nabla\bar\varphi_*|^2,\bar\phi) \nabla\bar\varphi_*\cdot {\bf{l}}ds=m_0,&\text{ for } x_3\in\mathbb R.
\end{cases}
\end{equation}

In the rest of this paper, we mainly consider the problems
\eqref{problem C2}--\eqref{limitcylinderproblem}. There are few remarks in order.
\begin{remark}
	It easy to see that the gravity force $\phi=gx_i$ $(i=1,2)$  satisfies \eqref{consevative force}.
	Also, it is easy to check the gravitational potential generated by the solid domain $\Omega^c$ (which is the complement of the fluid domain $\Omega$), i.e.,
	$$
	\phi(x)=\int_{\Omega'}\frac{\rho_s(y)}{|x-y|} d y
	$$
	satisfies the conditions  \eqref{consevative force}, where $x\in\Omega$ and $\rho_s\in L^1(\Omega')$ means the density distribution in $\Omega'$ is of finite mass. Similarly,  $\phi$ can also be the electric field.
\end{remark}

\begin{remark}
By \eqref{consevative force} and the Gagliardo-Nirenberg interpolation inequality, one has $\phi\in W^{1,q}(\Omega)$ for  $q>6$. Hence $\phi$ is an $L^{\infty}$ function by the Morrey's  inequality without extra condition.
\end{remark}

\begin{remark}For incompressible flows with external forces, the convergence rates \eqref{Icylinder boundary} and \eqref{Ibdecay}  are similar to the compressible Euler flows in \cite{MaLei} without the force. It follows from the convergence rates obtained in Theorem \ref{incompressible case} that the forces do not affect the convergence rates of velocity in the incompressible flows. However,  the external force plays an important role in the convergence rates of the velocity field of the compressible flows in {Theorem \ref{Theorem convergence rate}}.
\end{remark}
\begin{remark} For the flows past a body, the order of the convergence rate of the velocity field at infinity is independent of the force \cite{GuandWang}, while the convergence rates of velocity field for subsonic flows in infinitely long nozzles depend on the force $\phi$, which is described in Theorem \ref{Theorem convergence rate}. It is worth to point out that cases \textit{(i)} and \textit{(ii)} of Theorem \ref{Theorem convergence rate} show the flows have the precise far field asymptotic behavior even when  $\phi$ does not admit convergence rate at far fields. Furthermore, there is $1$ order loss of convergence rate in the polynomial case \textit{(ii)}, which is different from the exponential case \textit{(i)} in Theorem \ref{Theorem convergence rate}. Finally, case  \textit{(iii)} shows the convergence rates of the velocity field matches the slower one between the rate of boundary and force $\phi$.
\end{remark}
\begin{remark}
	The convergence rates of the flow velocity at far fields is independent of the obstacle in nozzle. This implies that Theorem \ref{Theorem convergence rate} can be applied to the flows through multidimensional nozzle studied in \cite{MR3638912}.
\end{remark}
Here we give the main ideas for the proof of the main results. Inspired by \cite{2019arXiv190101048L}, the existence of incompressible and compressible subsonic flows in the domain $\Omega$ is first established via the variational method. Then some uniform estimates are obtained which also implies that the order of low Mach number limit is $\epsilon^2$. The regularity of the solutions is improved since the corresponding subsonic potential equation is elliptic. As long as there is an  external force, both the boundary effect and the behavior of the external force at far fields have the strong influence on the convergence rates for the flows at far fields. Inspired by the delicate choice of weight function in\cite{MaLei, Ole1981ON}, $L^2-$norm of gradients of potential is obtained when  the boundary effect and the external force effect are combined together. Finally,  $L^\infty-$norm of the gradients of the potential established via the Nash-Moser iteration.

The rest of this paper is organized as follows. In Section \ref{approximate problems},  the approximate problems are introduced and  the variational method is applied to get the existence of weak solutions for both the incompressible and compressible flows. Moreover, some uniform estimates and the regularity of weak solution are obtained. The existence and uniqueness of incompressible-compressible different function are established in Sections \ref{existence and uniqueness}. In Section \ref{low mach number limit}, the proof of the low Mach number limit is given. In the last section, the convergence rates of velocity at far fields are established.

\section{Approximate Problems And Variational Approach}\label{approximate problems}
Since the domain is unbounded,  the truncation of the domain is used to to study problems \eqref{problemI2} and \eqref{problem C2}.

For any sufficiently large positive number $L$ and any set $U$, denote
\begin{equation*}
\Omega_L=\Omega\cap\big\{|x_3|<L\big\}\quad\text{and} \quad\dashint_Ufdx=\frac{1}{|U|}\int_Ufdx,\quad\end{equation*} where $f\in L^1(U)$. Later on, denote
\begin{equation*}\overline{S}=\inf\limits_{t\in\mathbb{R}}|{\Sigma}_{t}| \quad \text{and}\quad \underline{{S}}=\sup\limits_{t\in\mathbb{R}}|{\Sigma}_{t}|.
\end{equation*}

Define the space
\begin{equation}
\mathcal H_L=\{\varphi\in H^1({\Omega}_L):\varphi=0\text{ on } \Sigma_{-L}\}.
\end{equation}
One can directly check that $\mathcal H_L$ is a Hilbert space under $H^1$ norm.

Now we study the following problem about the incompressible flows in the truncated domain $\Omega_L$
\begin{equation}\label{problemI3}
\begin{cases}
\triangle\bar\varphi_L=0,\quad&\text{in $\Omega_L$},\\
\frac{\partial\bar\varphi_L}{\partial{\bf n}}=0,\quad&\text{on $\partial\Omega_L$},\\
\frac{\partial\bar\varphi_L}{\partial{x_3}}=\frac{m_0}{|\Sigma_L|},\quad&\text{on $\Sigma_L$},\\
\bar\varphi_L=0,\quad&\text{on $\Sigma_{-L}$}.
\end{cases}
\end{equation}
$\bar\varphi_L$ is called a weak solution of the problem \eqref{problemI3} in $\mathcal H_L$ if
\begin{equation}\label{weak_solutionI3}
\int_{{\Omega}_L}\nabla\bar\varphi_L\cdot\nabla\psi dx-\frac{m_0}{|\Sigma_{L}|}\int_{\Sigma_{L}}\psi dx'=0,\quad\text{for any $\psi\in \mathcal H_L.$}
\end{equation}

Define
\begin{equation}\label{functional}
\mathcal{I}_L(\psi)=\frac{1}{2}\int_{{\Omega}_L}|\nabla\psi|^2dx-\frac{m_0}{|\Sigma_{L}|}\int_{\Sigma_{L}}\psi dx',
\end{equation}
where $x'=(x_1,x_2)$.
The straightforward computations show that if $\bar\varphi_L$ is a minimizer of $\mathcal I_L$, $i.e.$,
\begin{equation}
\mathcal I_L(\bar\varphi_L)=\min_{\psi\in\mathcal H_L}\mathcal I_L(\psi),
\end{equation}
then $\bar\varphi_L$ must satisfy \eqref{weak_solutionI3}.

First,  the existence of minimizer of $\mathcal I_L$ and the basic estimate for the minimizer are obtained in the following lemma.
\begin{lemma}
	For any sufficiently large $L>0$, $\mathcal I_L(\psi)$ has a minimizer $\bar\varphi_L\in \mathcal H_L$. Moreover, the following estimate holds,
	\begin{equation}\label{key_estimae}
	\dashint_{{\Omega}_L}|\nabla\bar\varphi_L|^2dx\leq Cm_0^2,
	\end{equation}
	where $C$ is a constant independent of $L$.
\end{lemma}
\begin{proof}
	Choose a subset $U_L\subset\Omega_L$ such that $U_L\cap\Omega'=\emptyset$ and $\partial U_L\setminus(\Sigma_{-L}\cup\Sigma_{L})$ is $C^{2,\alpha}$ (see Figure. 1). Denote  $\mathscr{C}_L=B_1(0)\times\{-L\leq x_3\leq L\}$.
	Then there exists an invertible $C^{2,\alpha}$ map $\mathcal T_L$: ${U}_L\rightarrow \mathscr{C}_L$, $x\rightarrow y$ satisfying\\
	(i) $\mathcal T_L(\partial{U}_L)=\partial \mathscr{C}_L$.\\
	(ii) For any $-L\leq k\leq L$, $\mathcal T_L({U}_L\cap\{x_3=k\})=B_1(0)\times\{y_3=k\}$.\\
	(iii)$ \|\mathcal T_L\|_{C^{2,\alpha}},\ \|\mathcal T_L^{-1}\|_{C^{2,\alpha}}\leq C$.\\
	\begin{center}
		\includegraphics[width=0.8\textwidth]{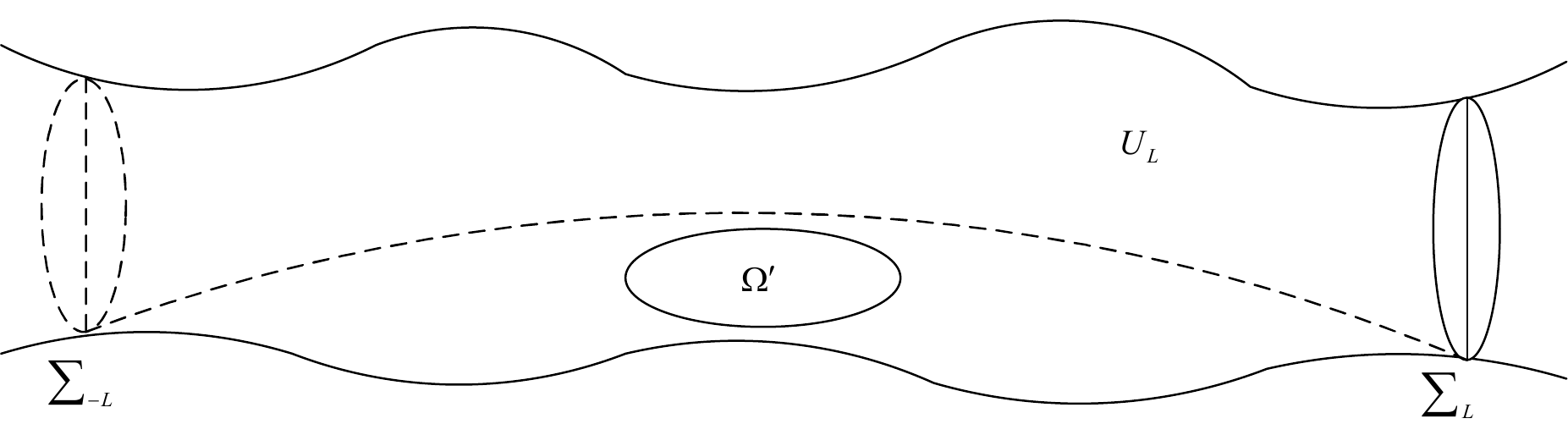}\\
		{\small Figure 1. Domain of $U_L$}
	\end{center}
The straightforward computations yield that
	\begin{equation}\label{S_Lestimate}
	\begin{split}
	\bigg|\int_{\Sigma_L}\psi dx'\bigg|&\leq C\int_{B_1(0)}|\psi(y',L)| dy'\leq C\int_{B_1(0)}\bigg(\int_{-L}^L|\partial_{y_3}\psi| dy_3\bigg)dy'\\
	&\leq C\int_{\mathscr{C}_L}|\nabla\psi|dy\leq C\int_{{U}_L}|\nabla\psi|dx\leq C\int_{{\Omega}_L}|\nabla\psi|dx.
	\end{split}
	\end{equation}
	Applying H$\ddot{\text{o}}$lder inequality gives
	\begin{equation}\label{estimate_on_S_L}
	\bigg|\int_{\Sigma_L}\psi dx'\bigg|\leq C |{\Omega}_L|^{\frac{1}{2}}\|\nabla\psi\|_{L^2(\Omega_L)}.
	\end{equation}
	The constant $C$ here and subsequently in the rest of the paper may change from line to line as long as what these constants depend on is clear.
	Substituting the estimate \eqref{estimate_on_S_L} into \eqref{functional} yields that
	\begin{equation}\label{coercive}
	\begin{split}
	\mathcal I_L(\psi)&=\int_{{\Omega}_L}|\nabla\psi|^2dx-\frac{m_0}{|\Sigma_{L}|}\int_{\Sigma_{L}}\psi dx'\\
	&\geq\int_{{\Omega}_L}|\nabla\psi|^2dx-C'\|\nabla\psi\|_{L^2({\Omega}_L)}\\
	&\geq\frac{1}{2}\|\nabla\psi\|_{L^2({\Omega}_L)}^2-C',
	\end{split}
	\end{equation}
	where $C'$ depends on $m_0$, $\overline{S}$, $\underline{{S}}$ and $|{\Omega}_L|$.
	This implies that the functional $\mathcal I_L(\psi)$ is coercive.
	Hence, as the same as the proof in \cite[Theorem 4]{MR2824469}, $\mathcal I_L(\psi)$ has a minimizer $\bar\varphi_L$ satisfying \eqref{key_estimae}. Moreover, the straight computations yield that
    $\bar\varphi_{L}$ is a weak solution of \eqref{problemI3}. Therefore, the proof is completed.
\end{proof}

From now on, denote  $\Omega(t_1,t_2)={\Omega}\cap\{t_1<x_3<t_2\}$.
Note that for any $f\in H^1(\Omega(t_1-1,t_2+1)$,
\begin{equation}\label{local estimate}
\bigg|\dashint_{\Omega(t_1-1,t_1)}f dx-  \dashint_{\Omega(t_2,t_2+1)}f dx\bigg|\leq C\int_{\Omega(t_1-1,t_2+1)}|\nabla f|dx,
\end{equation}
where $C$ is a constant which depends on $\Omega$ but is independent of $t_1$ and $t_2$. For the detailed proof of \eqref{local estimate}, one may refer to \cite[Proposition 4]{MR2824469}.
Furthermore,  the following Poincar$\acute{e}$ inequality
\begin{equation}\label{poin}
\bigg\|f(x)-\dashint_{\Omega(t,t+1)}f(x)dx\bigg\|_{L^p(\Omega(t,t+1))}\leq C\|\nabla f(x)\|_{L^p(\Omega(t,t+1))},
\end{equation}
holds \cite[Section 5.8 Theorem 1]{Evans2010Partial}, where $t\in \mathbb R$, $p\in[1,+\infty)$ and $C$ is independent of $t$.

In order to study the compressible flows,  the subsonic coefficients truncations are also needed to avoid the degeneracy of the equation near the sonic points. For $0<\epsilon_0<1$ and $0<\theta<1$, there exists a $q^\epsilon_{\theta}(\phi)$ such that $M^\epsilon(\phi)<\theta$. Set
$\mathring q^{\epsilon_0}_{\theta}(\phi)=\inf\limits_{0<\epsilon<\epsilon_0}q_\theta^\epsilon(\phi)$.
Denote
\begin{equation}
\hat q(q^2,\phi)=\begin{cases}q^2-2\phi\quad &\text{if $|q|\leq \mathring q^{\epsilon_0}_\theta(\phi)$},\\
\text{monotone smooth function}\quad&\text{if $\mathring q_\theta^{\epsilon_0}<|q|\leq\mathring q_{\frac{\theta+1}{2}}^{\epsilon_0}$},\\
\sup\limits_{x\in\Omega}\bigg(\big(\mathring q_{\frac{\theta+1}{2}}^{\epsilon_0}\big)^2(\phi)-2\phi\bigg)&\text{if $|q|>\mathring q_{\frac{\theta+1}{2}}^{\epsilon_0}$}.
\end{cases}
\end{equation}
Define
\begin{equation}\hat\rho^\epsilon(\mathcal G,\phi)=\tilde h^{-1}\bigg(\tilde h(1)-\frac{\epsilon^2\hat q(\mathcal G,\phi)}{2}\bigg).
\end{equation}

We first study the following problem with the subsonic truncations,
\begin{equation}\label{problemC3}
\begin{cases}
\Div\big(\hat\rho^\epsilon(|\nabla\varphi^\epsilon|^2,\phi)\nabla\varphi^\epsilon\big)=0,&\text{ in }\Omega,\\
\frac{\partial\varphi^\epsilon}{\partial{\bf n}}=0, &\text{ on } \partial\Omega,\\
\int_{\Sigma_t}\rho^\epsilon \nabla\varphi^\epsilon\cdot {\bf{l}}ds=m_0.
\end{cases}
\end{equation}
Denote
\begin{equation*}
\hat\rho_{\mathcal G}=\frac{\partial \hat\rho}{\partial \mathcal G }\quad\text{and}\quad\hat\rho_\phi=\frac{\partial\hat\rho}{\partial \phi}.
\end{equation*}
By the straightforward calculations,  the equation in  $\eqref{problemC3}$ can be written as
\begin{equation*}
\hat{a}_{ij}(\nabla\varphi^\epsilon,\phi)\partial_{ij}\varphi^\epsilon+\hat{b}_i(\nabla\varphi^\epsilon,\phi)\partial_{i}\varphi^\epsilon=0,
\end{equation*}
where
\begin{equation*}
\hat a_{ij}(\nabla\varphi^\epsilon,\phi)=\hat\rho^\epsilon(|\nabla\varphi^\epsilon|^2,\phi)\delta_{ij}+2\hat\rho_{\mathcal G}^\epsilon(|\nabla\varphi^\epsilon|^2,\phi)\partial_i\varphi^\epsilon\partial_j\varphi^\epsilon\quad\text{and}\quad \hat{b}_i(\nabla\varphi^\epsilon,\phi)=\hat{\rho}_{\phi}\partial_i\phi.
\end{equation*}
The directly calculations give that
\begin{equation}\label{elliptic coefficients}
\lambda\xi^2\leq\hat a_{ij}\xi_i\xi_j\leq\Lambda\xi^2\quad\text{ for $\xi\in\mathbb R^3$}\quad\text{and}\quad
 |\hat{b}_i(\nabla\varphi^\epsilon,\phi)|\leq C|\partial_i\phi|,
\end{equation}
where constants $\lambda$, $\Lambda$ and $C$ are independent of $\varphi^{\epsilon}$.

Since the domain $\Omega$ is unbounded, one can first study the following problem in truncated domains,
\begin{equation}\label{problemC4}
\begin{cases}
\Div(\hat\rho^\epsilon(|\nabla\varphi_L^\epsilon|^2,\phi)\nabla\varphi_L^\epsilon)=0,\quad&\text{in $\Omega_L$},\\
\frac{\partial\varphi_L^\epsilon}{\partial{\bf n}}=0,\quad&\text{on $\partial\Omega_L$},\\
\frac{\partial\varphi_L^\epsilon}{\partial{x_3}}=\frac{m_0}{|\Sigma_L|},\quad&\text{on $\Sigma_L$},\\
\varphi_L^\epsilon=0,\quad&\text{on $\Sigma_{-L}$}.
\end{cases}
\end{equation}
Denote
\begin{equation}
G^\epsilon(\mathcal G,\phi)=\frac{1}{2}\int_{0}^{\mathcal{G}}\hat\rho^\epsilon(\alpha,\phi)d\alpha.
\end{equation}
For  given solution $\bar\varphi_L$ of \eqref{problemI3} and the potential $\phi$ of the force, define
\begin{equation}
\mathcal L^\epsilon(\psi)=\epsilon^{-4}\int_{\Omega_L}G^\epsilon(|\nabla\psi|^2,\phi)-G^\epsilon(|\nabla\bar\varphi_L|^2,\phi)-\nabla\bar\varphi_L\cdot(\nabla\psi-\nabla\bar\varphi_L)dx\quad\text{for $\psi\in\mathcal H_L$},
\end{equation}
and  $\mathcal J^{\epsilon}(\tilde\psi)=\mathcal L^{\epsilon}(\bar\varphi_L+\epsilon^2\tilde\psi)$ for all $\tilde\psi\in\mathcal H_L$.
\begin{lemma}\label{lemma_local estimate} $\mathcal J^\epsilon(\tilde\psi)$ admits a unique minimizer $\tilde\varphi^\epsilon_L\in\mathcal H_L$ and 	 $\varphi^\epsilon_L$ is a weak solution of \eqref{problemC4} where
 \begin{equation}\label{tilde_varphi}
	\varphi_L^\epsilon=\bar\varphi_L+\epsilon^2\tilde\varphi_L^\epsilon.\end{equation}
Moveover, for any $t\in\big(\frac{-L}{4},\frac{L}{4}\big)$, one has
	\begin{equation}\label{local_estimate}
	\dashint_{\Omega(t,t+1)}|\nabla\bar\varphi_L|^2+|\nabla\tilde\varphi_L^\epsilon|^2dx\leq Cm_0^2,
	\end{equation}
	where $C$ is independent of $L$ and $\epsilon$.
\end{lemma}
The proof of Lemma \ref{lemma_local estimate} is the same to \cite{{2019arXiv190101048L}}, so we omit the details here.

Since $\bar\varphi_L$ and $\varphi^\epsilon_L$ are weak solutions of quasilinear elliptic equations of divergence form, similar to \cite[Lemmas 6 and 7 ]{MR2824469} and \cite[Lemma 4.5]{2019arXiv190101048L}, using the Nash-Moser iteration yields that there exists a positive constant $K'<\frac{L}{4}$ such that
\begin{equation}\label{regularity}
\|\nabla\tilde\varphi_L\|_{C^{0,\alpha}(\Omega(-K',K'))}+\|\nabla\bar\varphi_L\|_{C^{0,\alpha}(\Omega(-K',K'))}+\|\nabla\varphi^\epsilon_L\|_{C^{0,\alpha}(\Omega(-K',K'))}\leq Cm_0,
\end{equation}
where $C$ is a constant independent of $K'$.

\section{The Existence And Uniqueness of The Solution In The Whole Domain }\label{existence and uniqueness}
For any  fixed $\bar{x}\in{\Omega}$, choose $L$ large enough such that $\bar{x}\in{\Omega}(\frac{-L}{4},\frac{L}{4})$.  With abuse notations, we still denote $\varphi_L^\epsilon-\varphi_L^\epsilon(\bar x)$ and $\bar\varphi_L-\bar\varphi_L(\bar x)$ by $\varphi_L^\epsilon$ and $\bar\varphi_L$, respectively. It  follows from  \eqref{regularity} that \begin{equation} \|\nabla\bar\varphi_L\|_{C^{0,\alpha}(\Omega(-K',K'))}+ \|\nabla\varphi^\epsilon_L\|_{C^{0,\alpha}(\Omega(-K',K'))}\leq Cm_0.\end{equation}

By the standard diagonal procedure,  there exist functions $\bar\varphi\in C^{1,\alpha}_{loc}(\Omega)$ and $\varphi^\epsilon\in C^{1,\alpha}_{loc}(\Omega)$ such that for some $\alpha'<\alpha$,
\begin{equation*}
\lim\limits_{n\rightarrow\infty}\|\bar\varphi_{L_n}-\bar\varphi\|_{C^{1,\alpha'}(\Omega(-K',K'))}=0\quad\text{and}\quad \lim\limits_{n\rightarrow\infty}\|\varphi^\epsilon_{L_n}-\varphi^\epsilon\|_{C^{1,\alpha'}(\Omega(-K',K'))}=0.
\end{equation*}
Therefore, $\bar\varphi$ and $\varphi^\epsilon=\bar\varphi+\epsilon^2\tilde\varphi^\epsilon$ are solutions of \eqref{problemI2} and \eqref{problemC3}, respectively.

Now we are in position to prove the uniqueness of the solutions. We use a method  different from that in \cite{MR2824469} and  some energy estimates obtained here are also useful for the proof of  the convergence rates of velocity fields in Section \ref{secrate}.
\begin{lemma}
	Assume $\Omega$ satisfies \eqref{tildeOmega} and \eqref{Omega'}. Let $\bar\varphi_1$ and $\bar\varphi_2$ be the two  solutions of \eqref{problemI2}. Let $ \varphi^\epsilon_1$ and $\varphi^\epsilon_2$ be the two uniformly subsonic solutions of \eqref{problemC3}. Then $\nabla\bar\varphi_1=\nabla\bar\varphi_2$ and $\nabla\varphi^\epsilon_1=\nabla\varphi^\epsilon_2$ in $\Omega$.
\end{lemma}
\begin{proof}Denote $\Phi=\varphi^\epsilon_1-\varphi^\epsilon_2$. Then $\Phi$ satisfies
	\begin{equation}\label{uniquenesseq}
	\begin{cases}
	\partial_i(\mathfrak a_{ij}\partial_j\Phi)=0,&\quad \text{in $\Omega$,}\\
	\frac{\partial\Phi}{\partial{\bf n}}=0,&\quad \text{on $\partial\Omega$,}
	\end{cases}
	\end{equation}
	where \begin{equation}\mathfrak a_{ij}=\int_0^1\hat\rho^\epsilon( \mathfrak q^2,\phi)\delta_{ij}+2\rho_{\mathcal G}(\mathfrak q ^2,\phi)(s\partial_j\varphi_1^\epsilon+(1-s)\partial_j\varphi_2^\epsilon)(s\partial_i\varphi_1^\epsilon+(1-s)\partial_i\varphi_2^\epsilon)ds
	\end{equation}
	with
	\begin{equation}
	\mathfrak q^2=|s\nabla\varphi^\epsilon_1+(1-s)\nabla\varphi_2^\epsilon|^2.
	\end{equation}
	The straightforward computations show that there exist two constants $\lambda$ and $\Lambda$ such that
	\begin{equation}\label{aij_elliptic}
	\lambda|\xi|^2\leq\mathfrak a_{ij}\xi_i\xi_j\leq\Lambda|\xi|^2\quad\text{for $\xi\in\mathbb{R}^3$}.
	\end{equation}
	Moreover, one can increase $\Lambda$ so that the following Poincar$\acute{e}$ inequality holds on each cross section,
	\begin{equation}\label{Poincare_inequality}
	\bigg\|\mathcal Z-\dashint_{{\Sigma}_t}\mathcal Z ds\bigg\|_{L^2({\Sigma}_t)}\leq\Lambda\|\nabla \mathcal Z\|_{L^2({\Sigma}_t)}\quad\text{for any $\mathcal Z\in H^2_{loc}(\Omega(t-\epsilon,t+\epsilon))$.}
	\end{equation}
	
	For any $t_1<t_2$, $\mathfrak h$ and $\beta$ are constants to be determined later.  Denote
	\begin{equation}\label{zeta}
	\zeta(x_3;t_1,t_2,,\beta,\mathfrak h)=
	\begin{cases}
	1  &  x_3\leq t_1-\mathfrak h,\\
	e^{\beta(x_3-t_1+\mathfrak h)} &   t_1-\mathfrak h< x_3\leq t_1,\\
	e^{\beta \mathfrak h} & t_1< x_3\leq t_2,\\
	e^{\beta \mathfrak h}\cdot e^{-\beta(x_3-t_2)}  &  t_2<x_3\leq t_2+\mathfrak h,\\
	1 & x_3 >t_2+\mathfrak h.
	\end{cases}
	\end{equation}
	Multiplying $\Phi(\zeta(x_3,t_1,t_2,\beta,\mathfrak h)-1)$ on both sides of \eqref{uniquenesseq} and taking integral on $\Omega(t_1-\mathfrak h,t_2+\mathfrak h)$ yield
	\begin{equation}
	\begin{split}
	&\quad\int_{\Omega(t_1-\mathfrak h,t_2+\mathfrak h)}\mathfrak a_{ij}\partial_i\Phi\partial_j\Phi(\zeta-1)dx+\int_{\Omega(t_1-\mathfrak h,t_2+\mathfrak h)}\mathfrak a_{i3}\partial_i\Phi\Phi\partial_3\zeta dx\\&=\int_{\partial\Omega\cap\overline{\Omega(t_1-\mathfrak h,t_2+\mathfrak h)}}\mathfrak a_{ij}\partial_j\Phi\Phi(\zeta-1)n_ids.
	\end{split}
	\end{equation}
	For the boundary term, one has
	\begin{equation}\label{unique boundary}
	\mathfrak a_{ij}\partial_j\Phi n_i=\big(\hat\rho(|\nabla\varphi^\epsilon_1|^2,\phi)\partial_i\varphi^\epsilon_1-\hat\rho(|\nabla\varphi^\epsilon_2|^2,\phi)\partial_i\varphi^\epsilon_2\big)\cdot n_i=0.
	\end{equation}
	Moreover, the conserved mass flux on each cross section implies
	\begin{equation}\label{unique cross section}
	\int_{\Sigma_t}\mathfrak a_{i3}\partial_i\Phi dx'=\int_{\Sigma_t}\hat\rho(|\nabla\varphi^\epsilon_1|^2,\phi)\partial_3\varphi^\epsilon_1-\hat\rho(|\nabla\varphi^\epsilon_2|^2,\phi)\partial_3\varphi^\epsilon_2dx'=0.
	\end{equation}
	Set $\tilde\eta(t)=\int_{\Sigma_t}\Phi dx'$.
	Combining \eqref{unique boundary} and \eqref{unique cross section} yields that
	\begin{equation}\begin{aligned}
	&\quad\lambda\int_{\Omega(t_1 -\mathfrak h,t_2+\mathfrak h)}|\nabla\Phi|^2(\zeta-1)dx\\
	\leq &-\int\limits_{[t_1-\mathfrak h,t_1]\cup[t_2,t_2+\mathfrak h]}\bigg(\frac{\tilde\eta(x_3)}{|\Sigma_{x_3}|}\partial_3\zeta\int_{\Sigma_{x_3}}\mathfrak a_{i3}\partial_i\Phi dx'\bigg)dx_3\\
	&-\int_{\Omega(t_1-\mathfrak h,t_1)\cup\Omega(t_2,t_2+\mathfrak h)}\mathfrak a_{i3}\partial_i\Phi\bigg(\Phi-\frac{\tilde\eta(x_3)}{|\Sigma_{x_3}|}\bigg)
	\partial_3\zeta dx\\
	\leq& \bigg(\int_{\Omega(t_1-\mathfrak h,t_1)\cup\Omega(t_2,t_2+\mathfrak h)}\bigg(\Phi-\frac{\tilde\eta(x_3)}{|\Sigma_{x_3}|}\bigg)^2(\partial_3\zeta)^2\zeta^{-1}dx\bigg)^{\frac{1}{2}}\bigg(\int_{\Omega(t_1-\mathfrak h,t_1)\cup\Omega(t_2,t_2+\mathfrak h)}(\mathfrak a_{i3}\partial_i\Phi)^2\zeta dx\bigg)^{\frac{1}{2}}.
	\end{aligned}\end{equation}
	It follows from \eqref{Poincare_inequality} that one has
	\begin{equation}\begin{split}
	&\quad\bigg(\int_{\Omega(t_1-\mathfrak h,t_1)\cup\Omega(t_2,t_2+\mathfrak h)}\bigg(\Phi-\frac{\tilde\eta(x_3)}{|\Sigma_{x_3}|}\bigg)^2(\partial_3\zeta)^2\zeta^{-1}dx\bigg)^{\frac{1}{2}}\\
	&=\bigg\{\int_{t_1-\mathfrak h}^{t_1}+\int_{t_2}^{t_2+\mathfrak h}(\partial_3\zeta)^2\zeta^{-1}\bigg[\int_{\Sigma_{x_3}}\bigg(\Phi-\frac{\tilde\eta(x_3)}{|\Sigma_{x_3}|}\bigg)^2dx'\bigg]dx_3\bigg\}^\frac{1}{2}\\
	&\leq\bigg[\int_{t_1-\mathfrak h}^{t_1}+\int_{t_2}^{t_2+\mathfrak h}\Lambda^2(\partial_3\zeta)^2\zeta^{-1}\bigg(\int_{\Sigma_{x_3}}|\nabla\Phi|^2dx'\bigg)dx_3\bigg]^\frac{1}{2}\\
	&\leq\bigg(\int_{\Omega(t_1-\mathfrak h,t_1)\cup\Omega(t_2,t_2+\mathfrak h)}\Lambda^2(\partial_3\zeta)^2\zeta^{-1}|\nabla\Phi|^2dx\bigg)^{\frac{1}{2}}.
	\end{split}\end{equation}
	Note that $\partial_3\zeta=\beta\zeta$ for $x_3\in[t_1-\mathfrak h,t_1]\cup[t_2,t_2+\mathfrak h]$, then
	\begin{equation}\label{Phit1t2}
	\lambda\int_{\Omega(t_1-\mathfrak h,t_2+\mathfrak h)}|\nabla\Phi|^2(\zeta-1)dx\leq\Lambda^2\beta\int_{\Omega(t_1\mathfrak h,t_1)\cup\Omega(t_2,t_2+\mathfrak h)}|\nabla\Phi|^2\zeta dx.
	\end{equation}
	Set $\beta=\frac{\lambda}{\Lambda^2 }$. For any $T>0$, let  $t_1=-T$, $t_2=T$ and $\mathfrak h=T$. Hence,
	\begin{equation}
	\int_{\Omega(-T,T)}|\nabla\Phi|^2dx\leq Ce^{-\beta T}\int_{\Omega(-2T,2T)}|\nabla\Phi|^2dx\leq 2C\overline{S}Te^{-\beta T}\max|\nabla\Phi|^2.
	\end{equation}
	Let $T\rightarrow+\infty$ yields $|\nabla\Phi|=0$ $a.e.$  in $\Omega$.
		This finishes the proof of uniqueness of compressible case. One can follow the above steps to obtain the uniqueness of the incompressible case  where $\mathfrak a_{ij}=\delta_{ij}$.
\end{proof}

\section{ Low Mach Number limit}\label{low mach number limit}
This section devotes to the proof of  Theorem \ref{Theorem low Mach limit}.

In Lemma \ref{lemma_local estimate}, when $\epsilon_0$ and $\theta$ given, the solution $\varphi^\epsilon$ satisfies
\begin{equation}
|\nabla\varphi^\epsilon(x)|=|\epsilon^2\nabla\tilde\varphi^\epsilon(x)+\nabla\bar\varphi(x)|\leq\max|\nabla\bar\varphi(x)|+C(\epsilon_0,\theta)\epsilon^2.
\end{equation}
Hence, there exists a constant $\epsilon_{0,\theta}\in(0,\epsilon_0)$ such that  for any $0<\epsilon<\epsilon_{0,\theta}$, $ |\nabla\varphi^\epsilon(x,\epsilon_0,\theta)|<\mathring q^{\epsilon_0}_\theta(\phi)$.
Moreover, $\{\epsilon_{0,\theta}\}$ is a non-decreasing sequence with respect to $\theta$ which has an upper bound $\epsilon_0$. Set
$\epsilon_{0,cr}=\sup\limits_{0<\theta<1}\epsilon_{0,\theta}$. Hence one has
\begin{equation}
|\nabla\varphi^\epsilon(x)|=|\epsilon^2\nabla\tilde\varphi^\epsilon(x)+\nabla\bar\varphi(x)|\leq\max|\nabla\bar\varphi(x)|+C(\epsilon_0,\theta)\epsilon^2<\mathring q^{\epsilon_0}_{cr}(\phi).
\end{equation} This implies $M^\epsilon(\phi)<1$. Therefore,  $\varphi^\epsilon$ is the solution of \eqref{problem C2}.
Denote $\epsilon_c=\sup\limits_{0<\epsilon_0<1}\epsilon_{0,cr}$. Obviously, for any $\epsilon\in(0,\epsilon_c)$, one has $0<M^\epsilon(\phi)<1$ and
\begin{equation}
\nabla\varphi^\epsilon=\nabla\bar\varphi+\epsilon^2\nabla\tilde\varphi^\epsilon
\quad\text{or} \quad
{\bf{u}^\epsilon}=\bar{\bf{u}}+\epsilon^2\tilde{\bf u}.
\end{equation}
For convergence rates of density and pressure, the proof is the same to \cite{Zhang-Wang} and we omit the details. This completes the proof of Theorem \ref{Theorem low Mach limit}.

\section{Convergence rate of the velocity field}\label{secrate}

In this section, we investigate the convergence rates of the flows at far fields if the boundary of nozzle $\Omega$ tends to a perfect cylinder. First,  the $L^2$ convergence is obtained by energy estimate. $L^\infty$ convergence is established via Nash-Moser iteration. The proof is divided into four steps.
\subsection{Energy estimate for the case where the nozzle boundary satisfies \eqref{Icylindercase}} Let $\varphi^\epsilon$ be the uniformly subsonic solution of \eqref{problem C2}, whose existence and uniqueness are established in Theorem \ref{Theorem low Mach limit}.  Let $\varphi_*$ satisfy \eqref{cylinderproblem} and $\bar\varphi_*=\bar qx_3$ satisfy \eqref{limitcylinderproblem}.  Since the boundary is actually the cylinder, the desired convergence rates of the velocity field is exponential.

Denote $\Psi=\varphi^\epsilon-\varphi_{*}$. The straightforward computations yield that $\Psi$ satisfies
\begin{equation}\label{cylinderPsieq}
\begin{cases}
\partial_i(a_{ij}\partial_j\Psi)=0, &\text{in} \ \Omega\cap \{x_3>K\}, \\
\frac{\partial\Psi}{\partial \mathbf n}=0,  &\text{on}\ \partial\Omega\cap \{x_3>K\},
\end{cases}
\end{equation}
where \begin{equation}
a_{ij}=\int_{0}^1\rho^\epsilon(\hat q^2,\phi)\delta_{ij}+2\rho^\epsilon_\mathcal G(\hat q^2,\phi)(s\partial_i\varphi^\epsilon+(1-s)\partial_i\varphi_*))(s\partial_j\varphi^\epsilon+(1-s)\partial_j\varphi_{*})ds
\end{equation}
with
\begin{equation}
\hat q^2=|s\nabla\varphi^\epsilon+(1-s)\nabla\varphi_*|^2.
\end{equation}
Obviously, $a_{ij}$ satisfies \eqref{aij_elliptic}.
Let $\zeta(x_3,t_1,t_2,\beta,\mathfrak h)$ be the function defined in \eqref{zeta}. Multiplying $\Psi(\zeta-1)$ on both sides of the equation on \eqref{cylinderPsieq} as same as the estimates \eqref{Phit1t2} one has
\begin{equation}
e^{\beta \mathfrak h}\int_{\Omega(t_1,t_2)}|\nabla\Psi|^2dx\leq\int_{\Omega(t_1-\mathfrak h,t_2+\mathfrak h)}|\nabla\Phi|^2dx.
\end{equation}
Choosing $t_1=T$, $t_2=T+1$ and $\mathfrak h=\frac{T}{2}$ yields
\begin{equation}
e^{\beta T}\int_{\Omega(T,T+1)}|\nabla\Psi|^2dx\leq\int_{\Omega(\frac{T}{2},\frac{3T}{2}+1)}|\nabla\Psi|^2dx\leq C(T+1).
\end{equation}
If $T>0$ is large enough, then there exists a constant $\bar \alpha$ such that
\begin{equation}
\int_{\Omega(T,T+1)}|\nabla\Psi|^2dx\leq{e^{-\bar\alpha T}}.
\end{equation}
For the incompressible flows, if the boundary satisfies \eqref{Icylinder boundary}, then $\overline\Psi=\bar\varphi-qx_3$ satisfies
\begin{equation}\label{IcylinderPsieq}
\begin{cases}
\triangle\overline\Psi=0, &\text{in} \ \Omega\cap \{x_3>K\}, \\
\frac{\partial\overline\Psi}{\partial \mathbf n}=0,  &\text{on}\ \partial\Omega\cap \{x_3>K\}.
\end{cases}
\end{equation} Following the same arguments with $a_{ij}=\delta_{ij}$ in \eqref{cylinderPsieq} to get the exponential decay, $i.e.$, \begin{equation}
\int_{\Omega(T,T+1)}|\nabla\overline\Psi|^2dx\leq{e^{-\bar\alpha_1 T}},
\end{equation}
where $\bar\alpha_1$ is a positive constant.
\subsection{The energy estimate when the nozzle boundary satisfies \eqref{Iboundary_algebratic_rateii} and the conservative force satisfies \eqref{force algebratic rate}}\label{2} In this case, since the decay rates of boundary and external force is algebraic, the one cannot expect the convergence rates are exponential. Actually, the $L^2$ decay can only be algebraically fast.

Obviously, the unit outer normal of $\partial \Omega$ is
\begin{equation*}
{\bf{n}}=(n_1,n_2,n_3)=\frac{1}{\sqrt{G}}\bigg(\cos\tau+\frac{\partial f_1}{\partial\tau}\frac{\sin\tau}{r}, \sin\tau-\frac{\partial f_1}{\partial\tau}\frac{\cos\tau}{r},-\frac{\partial f_1}{\partial x_3}\bigg),
\end{equation*}
where \begin{equation*}G=1+\bigg(\frac{\partial f_1}{\partial \tau}\bigg)^2\frac{1}{r^2}+\bigg(\frac{\partial f_1}{\partial x_3}\bigg)^2.\end{equation*}Denote $\mathcal U=\varphi^\epsilon-\bar\varphi_*$. Then $\mathcal U$ satisfies
\begin{equation}\label{mathcalU}
\begin{cases}
\partial_i(a_{ij}\partial_j\mathcal U+b_i)=0, &\text{in} \ \Omega\cap \{x_3>K\}, \\
\frac{\partial\mathcal U}{\partial \mathbf n}=-\bar qn_3,  &\text{on}\ \partial\Omega\cap \{x_3>K\},
\end{cases}
\end{equation}
where \begin{equation}
a_{ij}=\int_{0}^1\rho^\epsilon(\mathfrak q^2,\mathcal E)\delta_{ij}+2\rho^\epsilon_\mathcal G(\mathfrak q^2,\mathcal E)(s\partial_i\varphi^\epsilon+(1-s)\partial_i\bar\varphi_*))(s\partial_j\varphi^\epsilon+(1-s)\partial_j\bar\varphi_*)ds
\end{equation}
and
\begin{equation}
b_i=\int_0^1\rho^\epsilon_\phi(\mathfrak q^2,\mathcal E )(\phi-\bar\phi)\big(s\partial_i\varphi^\epsilon+(1-s)\partial_i\bar\varphi_*\big)ds
\end{equation}
with
\begin{equation}
\mathfrak q^2=|s\nabla\varphi^\epsilon+(1-s)\nabla\bar\varphi_*|^2\quad\text{and}\quad\mathcal E=s\phi+(1-s)\bar\phi.
\end{equation}
And $a_{ij}$ satisfies \eqref{aij_elliptic}.
The straightforward computations yield that \begin{equation}\label{bi}
|b_i|_{C^1(\Sigma_{x_3})}\leq\frac{C}{x_3^{\mathfrak b_1}}.
\end{equation}
In fact, it follows from \eqref{Iboundary_algebratic_rateii} that one has
\begin{equation}\label{aijni}
\begin{split}
&\quad(a_{ij}\partial_j\mathcal U+b_i)n_i=\big(\rho^\epsilon(|\nabla\varphi^\epsilon|^2,\phi)\partial_i\varphi^\epsilon-\rho^\epsilon(|\nabla\bar\varphi_*|^2,\bar\phi)\partial_i\bar\varphi_*\big)n_i\\&=-\rho^\epsilon(|\bar q|^2,\bar\phi)\bar qn_3\leq \frac{C}{x_3^{\mathfrak a_1+1}}.
\end{split}\end{equation}
Moreover, on each cross section, one has
\begin{equation}\label{crosssectionestimate}\begin{split}
&\quad\int_{\Sigma_{x_3}}a_{i3}\partial_i{\mathcal U}+b_3dx'=\int_{\Sigma_{x_3}}\rho^\epsilon(|\nabla\varphi^\epsilon|^2,\phi)\partial_3\varphi^\epsilon-\rho^\epsilon(|\nabla\bar\varphi_*|^2,\bar\phi)\partial_3\bar\varphi_*dx'\\&\leq\int_{\Sigma_{x_3}}\rho^\epsilon(|\nabla\varphi^\epsilon|^2,\phi)\partial_3\varphi^\epsilon dx'-\int_{\Sigma_{x_3}}\rho^\epsilon(|\nabla\bar\varphi_*|^2,\bar\phi)\partial_3\bar\varphi_*dx'\\\quad&\quad+\int_{B_1(0)}\rho^\epsilon(|\nabla\bar\varphi_*|^2,\bar\phi)\partial_3\bar\varphi_*dx'-\int_{B_1(0)}\rho^\epsilon(|\nabla\bar\varphi_*|^2,\bar\phi)\partial_3\bar\varphi_*dx'\\&\leq C\bigg||\Sigma_{x_3}|-|B_1(0)|\bigg|\leq Cx_3^{-\mathfrak a_1}.\end{split}
\end{equation}

Choosing $t_1=T$ and $t_2=t_1+\bar K$ with $\bar K$ a positive integer to be determined later. Let $\zeta(x_3,t_1,t_2,\hat\beta,1)$ be the function define in \eqref{zeta} with $\mathfrak h=1$ and $\hat\beta$ to be determined later. Denote
\begin{equation}
s_1=\dashint_{\Omega(t_1-1,t_1)}\mathcal Udx\quad\text{and}\quad s_2=\dashint_{\Omega(t_2,t_2+1)}\mathcal Udx.
\end{equation}
Let
\begin{equation*}
\hat{\mathcal U}(x;t_1,t_2,s_1,s_2)=\left\{
\begin{array}{llr}
\mathcal U(x)-s_1,  & {x_3<t_1,}\\
\mathcal U(x)-s_1-\frac{s_2-s_1}{t_2-t_1}(x_3-t_1), & {t_1\leq x_3\leq t_2,}\\
\mathcal U(x)-s_2,  & {x_3>t_2}.
\end{array} \right.
\end{equation*}
Multiplying $\hat{\mathcal U}(\zeta-1)$ on both sides of \eqref{mathcalU} and integrating on $\Omega(t_1-1,t_2+1)$ yield that
\begin{equation}
\begin{split}
&\quad\int_{\overline{\Omega(t_1-1,t_2+1)}\cap\partial\Omega}(a_{ij}\partial_i\mathcal U+b_i)n_i\hat{\mathcal U}(\zeta-1)ds\\&=\int_{\Omega(t_1-1,t_2+1)}(a_{i3}\partial_i{\mathcal U}+b_3)\hat{\mathcal U}\partial_3\zeta dx+\int_{\Omega(t_1-1,t_2+1)}a_{ij}\partial_i\mathcal U\partial_j\mathcal U(\zeta-1)dx\\\quad&\quad+\int_{\Omega(t_1-1,t_2+1)}b_i(\zeta-1)\partial_i\mathcal U dx-\int_{\Omega(t_1,t_2)}(a_{i3}\partial_i\mathcal U+b_3)(\zeta-1)\frac{s_2-s_1}{t_2-t_1}dx,
\end{split}
\end{equation}
where $\zeta-1=0$ at $x_3=t_1-1$ and $x_3=t_2+1$ has been used. Then the direct calculations give
\begin{equation}\begin{split}
&\lambda\int_{\Omega(t_1-1,t_2+1)}|\nabla\mathcal U|^2(\zeta-1)dx\\
\leq&-\int_{\Omega(t_1-1,t_1)}(a_{i3}\partial_i{\mathcal U}+b_3)\hat{\mathcal U}\partial_3\zeta dx-\int_{\Omega(t_2,t_2+1)}(a_{i3}\partial_i{\mathcal U}+b_3)\hat{\mathcal U}\partial_3\zeta dx\\
\quad&-\int_{\Omega(t_1-1,t_2+1)}b_i(\zeta-1)\partial_i\mathcal U dx+\int_{\Omega(t_1,t_2)}(a_{i3}\partial_i\mathcal U+b_3)(\zeta-1)\frac{s_2-s_1}{t_2-t_1}dx\\\quad&+\int_{\overline{\Omega(t_1-1,t_2+1)}\cap\partial\Omega}(a_{ij}\partial_i\mathcal U+b_i)n_i\hat{\mathcal U}(\zeta-1)ds=\sum_{k=1}^5I_k.
\end{split}\end{equation}
 Now  $I_k\ (k=1,2\cdots,5)$ can be estimated one by one. Set \begin{equation*}
\chi_1(x_3)=\int_{\Sigma_{x_3}}(\mathcal U-s_1)dx'\quad\text{and}\quad\chi_2(x_3)=\int_{\Sigma_{x_3}}(\mathcal U-s_2)dx'.
\end{equation*}
The straightforward computations yield that
\begin{equation}
\begin{split}
|I_1|&\leq\bigg|\int_{\Omega(t_1-1,t_1)}(a_{i3}\partial_i\mathcal U+b_3)\partial_3\zeta\bigg(\hat{\mathcal U}-\frac{\chi_1(x_3)}{|\Sigma_{x_3}|}\bigg)dx+\int_{\Omega(t_1-1,t_1)}(a_{i3}\partial_i\mathcal U+b_3)\partial_3\zeta\frac{\chi_1(x_3)}{|\Sigma_{x_3}|}dx\bigg|\\
&\leq\bigg|\int_{\Omega(t_1-1,t_1)}a_{i3}\partial_i\mathcal U\partial_3\zeta\bigg(\hat{\mathcal U}-\frac{\chi_1(x_3)}{|\Sigma_{x_3}|}\bigg)dx\bigg|+\bigg|\int_{\Omega(t_1-1,t_1)}b_3\partial_3\zeta\bigg(\hat{\mathcal U}-\frac{\chi_1(x_3)}{|\Sigma_{x_3}|}\bigg)dx\bigg|\\
\quad&\quad+\bigg|\int_{\Omega(t_1-1,t_1)}(a_{i3}\partial_i\mathcal U+b_3)\partial_3\zeta\frac{\chi_1(x_3)}{|\Sigma_{x_3}|}dx\bigg|= \sum_{k=1}^{3}I_{1k}.
\end{split}\end{equation}
Noting that $(\partial_3\zeta)^2=\hat\beta^2\zeta^2$ when $x_3\in[t_1-\mathfrak h,t_1]\cup[t_2,t_2+\mathfrak h]$. It follows from the H$\ddot {\text o}$lder inequality and Poincar$\acute{e}$ inequality \eqref{Poincare_inequality} that one has
\begin{equation}\label{pi1}\begin{split}
|I_{11}|&\leq\bigg[\int_{\Omega(t_1-1,t_1)}(a_{i3}\partial_i\mathcal U)^2\zeta dx\bigg]^\frac{1}{2}\cdot\bigg[\int_{\Omega(t_1-1,t_1)}\bigg(\hat{\mathcal U}-\frac{\chi_1(x_3)}{|\Sigma_{x_3}|}\bigg)^2(\partial_3\zeta)^2\zeta^{-1}dx\bigg]^\frac{1}{2}\\
&\leq\bigg[\int_{\Omega(t_1-1,t_1)}\Lambda^2|\nabla\mathcal U|^2\zeta dx\bigg]^\frac{1}{2}\cdot\bigg[\int_{t_1-1}^{t_1}\int_{\Sigma_{x_3}}\bigg(\hat{\mathcal U}-\frac{\chi_1(x_3)}{|\Sigma_{x_3}|}\bigg)^2dx'(\partial_3\zeta)^2\zeta^{-1} dx_3\bigg]^\frac{1}{2}
\\
&\leq\bigg[\int_{\Omega(t_1-1,t_1)}\Lambda^2|\nabla\mathcal U|^2\zeta dx\bigg]^\frac{1}{2}\cdot\bigg[\int_{t_1-1}^{t_1}\int_{\Sigma_{x_3}}\Lambda^2|\nabla\mathcal U|^2dx'(\partial_3\zeta)^2\zeta^{-1} dx_3\bigg]^\frac{1}{2}\\
&\leq\Lambda^2\hat\beta\int_{\Omega(t_1-1,t_1)}|\nabla\mathcal U|^2\zeta dx.
\end{split}\end{equation}
Using \eqref{force algebratic rate} gives
\begin{equation}\label{pi2}
\begin{split}
|I_{12}|&\leq\bigg(\int_{\Omega(t_1-1,t_1)}b_3^2(\partial_3\zeta)^2dx\bigg)^\frac{1}{2}\cdot\bigg(\int_{t_1-1}^{t_1}\int_{\Sigma_{x_3}}\bigg(\hat{\mathcal U}-\frac{\chi_1(x_3)}{|\Sigma_{x_3}|}\bigg)^2dx' dx_3\bigg)^\frac{1}{2}\\
&\leq\frac{C\Lambda}{(t_1-1)^{\mathfrak b_1}}\bigg(\frac{\hat\beta}{2}\big(e^{2\hat\beta}-1\big)\bigg)^\frac{1}{2}\bigg(\int_{\Omega(t_1-1,t_1)}|\nabla\mathcal U|^2dx\bigg)^\frac{1}{2}\\
&\leq\frac{C\hat\beta\big(e^{2\hat\beta}-1\big)}{\epsilon_1(t_1-1)^{2\mathfrak b_1}}+\frac{\epsilon_1}{2}\int_{\Omega(t_1-1,t_1)}|\nabla\mathcal U|^2dx,
\end{split}\end{equation}
where $\epsilon_1$ is a small positive constant to be determined later. Applying the estimate \eqref{crosssectionestimate} yields
\begin{equation}\label{pi3}
\begin{split}
|I_{13}|&\leq\bigg|\int_{t_1-1}^{t_1}\bigg(\int_{\Sigma_{x_3}}(a_{i3}\partial_i\mathcal U+b_3)dx'\bigg)\partial_3\zeta\frac{\chi_1(x_3)}{|\Sigma_{x_3}|}dx_3\bigg|\\
&\leq\bigg[\int_{t_1-1}^{t_1}\bigg(\int_{\Sigma_{x_3}}(a_{i3}\partial_i\mathcal U+b_3)dx'\frac{\partial_3\zeta}{|\Sigma_{x_3}|}\bigg)^2dx_3\bigg]^\frac{1}{2}\cdot\bigg[\int_{t_1-1}^{t_1}\chi_1^2(x_3)dx_3\bigg]^\frac{1}{2}\\
&\leq\frac{C}{(t_1-1)^\mathfrak {a_1}}\bigg(\frac{\hat\beta}{2}\big(e^{2\hat\beta}-1\big)\bigg)^\frac{1}{2}\bigg[\int_{t_1-1}^{t_1}\bigg(\int_{\Sigma_{x_3}}(\mathcal U-s_1)^2dx'\bigg)dx_3\bigg]^\frac{1}{2}\\
&\leq\frac{C}{(t_1-1)^\mathfrak {a_1}}\bigg(\frac{\hat\beta}{2}\big(e^{2\hat\beta}-1\big)\bigg)^\frac{1}{2}\bigg(\int_{\Omega(t_1-1,t_1)}|\nabla\mathcal U|^2dx\bigg)^\frac{1}{2}\\
&\leq\frac{C\hat\beta\big(e^{2\hat\beta}-1\big)}{\epsilon_1(t_1-1)^{2\mathfrak a_1}}+\frac{\epsilon_1}{2}\int_{\Omega(t_1-1,t_1)}|\nabla\mathcal U|^2dx.
\end{split}
\end{equation}
Therefore, combining \eqref{pi1}, \eqref{pi2}, and \eqref{pi3} together shows
\begin{equation}\label{I1}
|I_1|\leq\Lambda^2\hat\beta\int_{\Omega(t_1-1,t_1)}|\nabla\mathcal U|^2\zeta dx+ \frac{C\hat\beta\big(e^{2\hat\beta}-1\big)}{\epsilon_1(t_1-1)^{2\mathfrak a_1}}+\frac{C\hat\beta\big(e^{2\hat\beta}-1\big)}{\epsilon_1(t_1-1)^{2\mathfrak b_1}}+\epsilon_1\int_{\Omega(t_1-1,t_1)}|\nabla\mathcal U|^2dx.
\end{equation}
Similarly, one has
\begin{equation}\label{I2}
|I_2|\leq\Lambda^2\hat\beta\int_{\Omega(t_2,t_2+1)}|\nabla\mathcal U|^2\zeta dx+ \frac{C\hat\beta\big(e^{2\hat\beta}-1\big)}{\epsilon_1(t_2)^{2\mathfrak a_1}}+\frac{C\hat\beta\big(e^{2\hat\beta}-1\big)}{\epsilon_1(t_2)^{2\mathfrak b_1}}+\epsilon_1\int_{\Omega(t_2,t_2+1)}|\nabla\mathcal U|^2dx.
\end{equation}
For $I_3$, it follows from \eqref{force algebratic rate} that one has
\begin{equation}\label{I3}
\begin{split}
|I_3|&\leq\frac{Ce^{\hat\beta}}{(t_1-1)^{\mathfrak b_1}}\int_{\Omega(t_1-1,t_2+1)}|\nabla\mathcal U|dx\\&\leq\frac{Ce^{\hat\beta}}{(t_1-1)^{\mathfrak b_1}}|\Omega(t_1-1,t_2+1)|^\frac{1}{2}\bigg(\int_{\Omega(t_1-1,t_2+1)}|\nabla\mathcal U|^2dx\bigg)^\frac{1}{2}\\&\leq\frac{Ce^{2\hat\beta}(t_2-t_1+2)}{\epsilon_1(t_1-1)^{2\mathfrak b_1}}+\frac{\epsilon_1}{2}\int_{\Omega(t_1-1,t_2+1)}|\nabla\mathcal U|^2dx.
\end{split}
\end{equation}
It follows from \eqref{crosssectionestimate} that
\begin{equation}\label{I4'}\begin{split}
|I_4|&=\bigg|\frac{s_2-s_1}{t_2-t_1}\int_{t_1}^{t_2}\bigg(\int_{\Sigma_{x_3}}a_{i3}\partial_i\mathcal U+b_3dx'\bigg)(\zeta-1)dx_3\bigg|\leq\frac{Ce^{\hat\beta}}{t_1^{\mathfrak a_1}}|s_2-s_1|.
\end{split}\end{equation}
With the aid of estimate \eqref{local estimate}, one has
\begin{equation}\label{I4}
|I_4|\leq \frac{Ce^{\hat\beta}}{t_1^{\mathfrak a_1}}\int_{\Omega(t_1-1,t_2+1)}|\nabla\mathcal U|dx\leq\frac{Ce^{2\hat\beta}(t_2-t_1+2)}{\epsilon_1t_1^{2\mathfrak a_1}}+\frac{\epsilon_1}{2}\int_{\Omega(t_1-1,t_2+1)}|\nabla\mathcal U|^2dx.
\end{equation}

For $I_5$, it follows from \eqref{aijni} that one has
\begin{equation}\label{I5'}
\begin{split}
|I_5|&\leq\frac{Ce^{\hat\beta}}{(t_1-1)^{\mathfrak a_1+1}}\int_{\partial\Omega\cap\overline{\Omega(t_1-1,t_2+1)}}\hat{\mathcal U}ds\\
&\leq\frac{Ce^{\hat\beta}}{(t_1-1)^{\mathfrak a_1+1}}\sum_{i=0}^{\bar K+1}\int_{\partial\Omega\cap\overline{\Omega(t_1-1+i,t_1+i)}}\hat{\mathcal U}ds\\
&\leq\frac{Ce^{\hat\beta}}{(t_1-1)^{\mathfrak a_1+1}}\sum_{i=0}^{\bar K+1}\bigg(\int_{\Omega(t_1-1+i,t_1+i)}|\nabla\hat{\mathcal U}|+|\hat{\mathcal U}|dx\bigg)\\
&\leq\frac{Ce^{\hat\beta}}{(t_1-1)^{\mathfrak a_1+1}}\bigg(\int_{\Omega(t_1-1,t_2+1)}|\nabla\mathcal U|dx+|s_2-s_1|+\int_{\Omega(t_1-1,t_2+1)}|\hat{\mathcal U}|dx\bigg).
\end{split}\end{equation}Define
\begin{equation}
\mathcal D_i=\dashint_{\Omega(t_1+i,t_1+i+1)}\mathcal U dx,\quad i=0,1,\cdots,\bar K-1.
\end{equation}
As same as the estimate \eqref{local estimate}, one has
\begin{equation}\label{Di-s1}
|\mathcal D_i-s_1|\leq\int_{\Omega(t_1-1,t_1+i+1)}|\nabla\mathcal U|dx.
\end{equation}
Therefore, the direct computations give
\begin{equation}\label{U-s1}
\begin{split}
&\quad\int_{\Omega(t_1,t_2)}|\mathcal U-s_1|dx\\
&\leq\sum_{j=0}^{\bar K-1}\int_{\Omega(t_1+j,t_1+j+1)}|\mathcal U-\mathcal D_j|+|\mathcal D_j-s_1|dx\\
&\leq C\sum_{j=0}^{\bar K-1}\int_{\Omega(t_1+j,t_1+j+1)}|\nabla\mathcal U|dx+\sum_{j=0}^{\bar K-1}\int_{\Omega(t_1+j,t_1+j+1)}\bigg(\int_{\Omega(t_1-1,t_1+j+1)}|\nabla\mathcal U|dx\bigg)dx\\
&\leq C\sum_{j=0}^{\bar K-1}\int_{\Omega(t_1+j,t_1+j+1)}|\nabla\mathcal U|dx+C\sum_{j=0}^{\bar K-1}\int_{\Omega(t_1-1,t_2+1)}|\nabla\mathcal U|dx\\
&\leq C(t_2-t_1+2)\int_{\Omega(t_1-1,t_2+1)}|\nabla\mathcal U|dx.
\end{split}
\end{equation}
This yields
\begin{equation}\label{I5''}
\begin{split}
&\quad\int_{\Omega(t_1-1,t_2+1)}|\hat{\mathcal U}|
dx=\int_{\Omega(t_1-1,t_1)}|\hat{\mathcal U}| dx+\int_{\Omega(t_2,t_2+1)}|\hat{\mathcal U}| dx+\int_{\Omega(t_1,t_2)}|\hat{\mathcal U}| dx\\[2mm]&
\leq C\bigg(\int_{\Omega(t_1-1,t_1)\cup\Omega(t_2,t_2+1)}|\nabla{\mathcal U}| dx\bigg)+\int_{\Omega(t_1,t_2)}\bigg|{\mathcal U}-s_1-\frac{s_2-s_1}{t_2-t_1}(x_3-t_1)\bigg|dx\\&
\leq C\bigg(\int_{\Omega(t_1-1,t_1)\cup\Omega(t_2,t_2+1)}|\nabla{\mathcal U}| dx\bigg)+|s_2-s_1|\int_{t_1}^{t_2}\int_{\Sigma_{x_3}}\frac{x_3-t_1}{t_2-t_1}dx'dx_3+\int_{\Omega(t_1,t_2)}|{\mathcal U}-s_1| dx\\
&\leq C\bigg(\int_{\Omega(t_1-1,t_1)\cup\Omega(t_2,t_2+1)}|\nabla{\mathcal U}| dx\bigg)+C(t_2-t_1)|s_2-s_1|+C(t_2-t_1+2)\int_{\Omega(t_1-1,t_2+1)}|\nabla{\mathcal U}|dx\\
&\leq C(t_2-t_1+2)\int_{\Omega(t_1-1,t_2+1)}|\nabla{\mathcal U}|dx.
\end{split}
\end{equation}
Therefore, the estimate \eqref{I5'}, together with \eqref{I5''}, implies
\begin{equation}\label{I5}\begin{split}
\quad|I_5|&\leq\frac{Ce^{\hat\beta}(t_2-t_1+2)}{(t_1-1)^{\mathfrak a_1+1}}\int_{\Omega(t_1-1,t_2+1)}|\nabla\mathcal U|dx\\
&\leq\frac{Ce^{\hat\beta}(t_2-t_1+2)}{(t_1-1)^{\mathfrak a_1+1}}|\Omega(t_1-1,t_2+1)|^\frac{1}{2}\bigg(\int_{\Omega(t_1-1,t_2+1)}|\nabla\mathcal U|^2dx\bigg)^\frac{1}{2}\\
&\leq \frac{Ce^{2\hat\beta}(t_2-t_1+2)^3}{\epsilon_1(t_1-1)^{2\mathfrak a_1+2}}+\frac{\epsilon_1}{2}\int_{\Omega(t_1-1,t_2+1)}|\nabla\mathcal U|^2dx.
\end{split} \end{equation}

Collecting \eqref{I1}, \eqref{I2}, \eqref{I3}, \eqref{I4}, and \eqref{I5} together gives
\begin{equation}\label{I}\begin{split}
&\quad\lambda\int_{\Omega(t_1-1,t_2+1)}|\nabla\mathcal U|^2(\zeta-1)dx\\
&\leq \Lambda^2\hat\beta\int_{\Omega(t_1-1,t_1)\cup\Omega(t_2,t_2+1)}|\nabla\mathcal U|^2\zeta dx+\epsilon_1\int_{\Omega(t_1-1,t_1)\cup\Omega(t_2,t_2+1)}|\nabla\mathcal U|^2dx\\[2mm]&\quad+\frac{3\epsilon_1}{2}\int_{\Omega(t_1-1,t_2+1)}|\nabla\mathcal U|^2dx+\frac{C\hat{\beta}(e^{2\hat\beta}-1)}{\epsilon_1(t_1-1)^{2\mathfrak a_1}}+\frac{C\hat{\beta}(e^{2\hat\beta}-1)}{\epsilon_1(t_1-1)^{2\mathfrak b_1}}+\frac{C\hat{\beta}(e^{2\hat\beta}-1)}{\epsilon_1t_2^{2\mathfrak a_1}}\\[2mm]&\quad+\frac{C\hat{\beta}(e^{2\hat\beta}-1)}{\epsilon_1t_2^{2\mathfrak b_1}}+\frac{Ce^{2\hat\beta}(t_2-t_1+2)}{\epsilon_1(t_1-1)^{2\mathfrak b_1}}+\frac{Ce^{2\hat\beta}(t_2-t_1+2)}{\epsilon_1(t_1)^{2\mathfrak a_1}}+\frac{Ce^{2\hat\beta}(t_2-t_1+2)^3}{\epsilon_1(t_1-1)^{2\mathfrak a_1+2}}.
\end{split}
\end{equation}
Choosing $\hat\beta=\frac{\lambda}{\Lambda^2}$ and taking $\mathfrak b=\min(\mathfrak a_1,\mathfrak b_1)$ yield
\begin{equation}\begin{split}
&\quad\lambda e^{\hat\beta}\int_{\Omega(t_1,t_2)}|\nabla\mathcal U|^2dx\\&\leq (\lambda+3\epsilon_1)\int_{\Omega(t_1-1,t_2+1)}|\nabla\mathcal U|^2dx+\frac{C(t_2-t_1+2)}{\epsilon_1(t_1-1)^{2\mathfrak b}}+\frac{C(t_2-t_1+2)^3}{\epsilon_1(t_1-1)^{2\mathfrak b+2}},
\end{split}\end{equation}
where $C$ is a constant depending on $\hat\beta$.

If one chooses $\epsilon_1$ small enough such that $\frac{\lambda+3\epsilon_1}{\lambda e^{\hat\beta}}\leq\mathfrak c_0<1$, then it holds that
\begin{equation}
\int_{\Omega(t_1,t_2)}|\nabla\mathcal U|^2dx\leq \mathfrak c_0\int_{\Omega(t_1-1,t_2+1)}|\nabla\mathcal U|^2dx+\frac{C(t_2-t_1+2)}{(t_1-1)^{2\mathfrak b}}+\frac{C(t_2-t_1+2)^3}{(t_1-1)^{2\mathfrak b+2}}.
\end{equation}
Choosing $\bar K$ large enough such that $\mathfrak c_0\frac{t_2-t_1+2}{t_2-t_1}<\mathfrak c<1$ yields
\begin{equation}\label{t2t1}\begin{split}
&\quad\frac{1}{t_2-t_1}\int_{\Omega(t_1,t_2)}|\nabla\mathcal U|^2dx\\&\leq\mathfrak c\frac{1}{t_2-t_1+2}\int_{\Omega(t_1-1,t_2+1)}|\nabla\mathcal U|^2dx+\frac{C(t_2-t_1+2)}{(t_2-t_1)(t_1-1)^{2\mathfrak b}}+\frac{C(t_2-t_1+2)^3}{(t_2-t_1)(t_1-1)^{2\mathfrak b+2}}.
\end{split}\end{equation}

Let $J$ be an integer satisfying $\frac{T}{2}-1\leq J<\frac{T}{2}$. Denote $t_{1,i}=T-i$ and $t_{2,i}=T+\bar K+i\ (i=0,1\cdots J)$. If $T>0$ is large enough, one  has
\begin{equation}
\frac{t_{2,i}-t_{1,i}+2}{t_{1,i}-1}\leq C \quad\text{for $i=0,1\cdots J$}.
\end{equation}
Substituting $t_{1,i}$ and $t_{2,i}$ into \eqref{t2t1} yields
\begin{equation}\label{t2it1i}
\frac{1}{t_{2,i}-t_{1,i}}\int_{\Omega(t_{1,i},t_{2,i})}|\nabla\mathcal U|^2dx\leq\mathfrak c\frac{1}{t_{2,i}-t_{1,i}+2}\int_{\Omega(t_{2,i}-t_{1,i}+2)}|\nabla\mathcal U|^2dx+\frac{C}{(t_{1,i}-1)^{2\mathfrak b}}.
\end{equation}
Iterating \eqref{t2it1i} gives
\begin{equation}
\frac{1}{t_{2,0}-t_{1,0}}\int_{\Omega(t_{1,0},t_{2,0})}|\nabla\mathcal U|^2dx\leq \mathfrak c^{J}\frac{1}{t_{2,J}-t_{1,J}}\int_{\Omega(t_{1,J},t_{2,J})}|\nabla\mathcal U|^2dx+\sum_{j=0}^{J}\mathfrak c^j\frac{C}{(t_{1,j})^{2\mathfrak b}}.
\end{equation}
Since $|\nabla\mathcal U|$ is bounded and $T$ is large, one has
\begin{equation}
\frac{1}{\bar K}\int_{\Omega(T,T+\bar K)}|\nabla\mathcal U|^2dx\leq C\mathfrak c^{J}+\frac{C}{T^{2\mathfrak b}}\leq\frac{C}{T^{2\mathfrak b}}.
\end{equation}
Thus,
\begin{equation}\label{energyestimate2}
\int_{\Omega(T,T+1)}|\nabla\mathcal U|^2dx\leq\int_{\Omega(T,T+\bar K)}|\nabla\mathcal U|^2dx\leq \frac{C\bar K}{T^{2\mathfrak b}}\leq\frac{C}{T^{2\mathfrak b}}.
\end{equation}

For the incompressible flows, when the nozzle boundary satisfies \eqref{Iboundary_algebratic_rateii}, then $\bar{\mathcal U}=\bar\varphi-qx_3$ satisfies
\begin{equation}\label{ImathcalU}
\begin{cases}
\triangle\bar{\mathcal U}=0, &\text{in} \ \Omega\cap \{x_3>K\}, \\
\frac{\partial\bar{\mathcal U}}{\partial \mathbf n}=- qn_3,  &\text{on}\ \partial\Omega\cap \{x_3>K\}.
\end{cases}
\end{equation}
Let $a_{ij}=\delta_{ij}$ and $b_i=0$ in \eqref{mathcalU}. Then \eqref{aijni} and \eqref{crosssectionestimate} can be written as
\begin{equation}
\int_{\Sigma_{x_3}}\partial_{x_3}\bar{\mathcal U}dx'\leq Cx_3^{-\mathfrak a_1}\quad\text{and}\quad\bigg|\frac{\partial\bar{\mathcal U}}{\partial {\bf n}}\bigg|\leq Cx_3^{-\mathfrak a_1-1}\quad \text{for $x_3>K$.}
\end{equation}
Similar to the proof for estimate \eqref{energyestimate2}, one can  conclude that
\begin{equation}\label{Ienergyestimate2}
\int_{\Omega(T,T+1)}|\nabla\bar{\mathcal U}|^2dx\leq\frac{C}{T^{2\mathfrak a_1}}.
\end{equation}
\subsection{Energy estimate where the boundary of the nozzle satisfies \eqref{Iboundary_algebratic_rateii} with $\mathfrak a_1>1$} In this case, the  velocity at the downstream is not constant. Then convergence rates of the velocity at the boundary in normal direction is $O(x_3^{-\mathfrak a_1})$. In order to get the convergence rates of velocity, we ask $\mathfrak a_1>1$.

Ler $\varphi_*$ is the uniformly subsonic solution of \eqref{cylinderproblem}. Obviously, $\nabla\varphi_*$ is not a constant. Denote $ W=\varphi^\epsilon-\varphi_*$. Then $ W$ satisfies
\begin{equation}\label{W}
\begin{cases}
\partial_i(a_{ij}\partial_jW)=0, &\text{in} \ \Omega\cap \{x_3>K\}, \\
\frac{\partial W}{\partial \mathbf n}=-\nabla\varphi_*\cdot{\bf n},  &\text{on}\ \partial\Omega\cap \{x_3>K\},
\end{cases}
\end{equation}
where we abuse the notations \begin{equation}
a_{ij}=\int_{0}^1\rho^\epsilon(\bar{\mathfrak q}^2,\phi)\delta_{ij}+2\rho^\epsilon_\mathcal G(\bar{\mathfrak q}^2,\phi)(s\partial_i\varphi^\epsilon+(1-s)\partial_i\varphi_*))(s\partial_j\varphi^\epsilon+(1-s)\partial_j\varphi_*)ds
\end{equation}
with
\begin{equation}
\bar{\mathfrak q}^2=|s\nabla\varphi^\epsilon+(1-s)\nabla\varphi_*|^2.
\end{equation}
Obviously $a_{ij}$ satisfies \eqref{aij_elliptic}. On each cross section $\Sigma_{x_3}$, one has
\begin{equation}\label{crosssectionestimate'}\begin{split}
&\quad\int_{\Sigma_{x_3}}a_{i3}\partial_i{W}dx
=\int_{\Sigma_{x_3}}\rho^\epsilon(|\nabla\varphi^\epsilon|^2,\phi)\partial_3\varphi^\epsilon-\rho^\epsilon(|\nabla\varphi_*|^2,\phi)\partial_3\varphi_*dx'\\&=\int_{\Sigma_{x_3}}\rho^\epsilon(|\nabla\varphi^\epsilon|^2,\phi)\partial_3\varphi^\epsilon dx'-\int_{\Sigma_{x_3}}\rho^\epsilon(|\nabla\varphi_*|^2,\phi)\partial_3\varphi_*dx'\\\quad&\quad+\int_{B_1(0)}\rho^\epsilon(|\nabla\varphi_*|^2,\phi)\partial_3\varphi_*dx'-\int_{B_1(0)}\rho^\epsilon(|\nabla\varphi_*|^2,\phi)\partial_3\varphi_*dx'\\&\leq C\bigg||\Sigma_{x_3}|-|B_1(0)|\bigg|\leq Cx_3^{-\mathfrak a_1}.\end{split}
\end{equation}
Writing $\bar{\bf{n}}=(\cos\tau,\sin\tau,0)$, then $\nabla\varphi_*\cdot\bar{\bf{n}}=0$.
In fact, on the boundary, one has
\begin{equation}\label{aijni'}
\begin{split}
&\quad(a_{ij}\partial_jW)n_i=\big(\rho^\epsilon(|\nabla\varphi^\epsilon|^2,\phi)\nabla\varphi^\epsilon-\rho^\epsilon(|\nabla\varphi_*|^2,\phi)\nabla\varphi_*\big)\cdot{\bf n}\\&=-\rho^\epsilon(|\nabla\varphi_*|^2,\phi)\nabla\varphi_*\cdot{\bf n}=-\rho^\epsilon(|\nabla\varphi_*|^2,\phi)\nabla\varphi_*\cdot({\bf n}-\bar{\bf n})\leq \frac{C}{x_3^{\mathfrak a_1}},
\end{split}\end{equation}
where the assumption \eqref{Iboundary_algebratic_rateii} is used.
Note that the boundary convergence rates in \eqref{aijni'} is $-\mathfrak a_1$. This is main difference between the current case and the case where there is no external force decay as in \eqref{aijni}. Therefore,  with the help of the decay \eqref{crosssectionestimate'},  the $L^2$ convergence rates  $-(\mathfrak a_1 -1)$ can be established.

It follows from \eqref{crosssectionestimate'}, \eqref{aijni'} and the the same strategy to get the estimate \eqref{energyestimate2} in Step 2 ($b_i=0$) that one has
\begin{equation}\label{energyestimate3}
\int_{\Omega(T,T+1)}|\nabla W|^2dx\leq\frac{C}{T^{2(\mathfrak a_1-1)}}.
\end{equation}
\subsection{$L^\infty-$estimate} Based on the $L^2$ estimates obtained in the last sections, $L^\infty-$ norm of velocity fields can be established via Nash-Moser iteration. Since the general case is that $b_i\not=0$ and the boundary estimate, we only consider the equation \eqref{mathcalU} in Section \ref{2} and first prove the estimate near the boundary.
For any point $ \tilde x=(\tilde x_1,\tilde x_{2},\tilde x_{3})\in\partial{\Omega}$ and $\tilde x_3>0$ sufficiently large, suppose that
$ \big(x_1(y_1,y_2),x_2(y_1,y_2),x_3(y_1,y_2)\big) \in C^{2,\alpha}$ is the standard parametrization of $\Omega$ in a small neighborhood of $\tilde x$. Then unit outer normal vector $\bf{n}$ satisfying
\begin{equation}
\cos({\bf{n},x_i})=n_i(y_1,y_2)\in C^{1,\alpha} \quad\text{for $i=1,2,3$}.
\end{equation}

Define the  map $\mathcal{M}_{y}:y\rightarrow x$ as follows
\begin{equation}
x_i= x_i(y_1,y_2)+y_3^{-1}\int_{y_1}^{y_1+y_3}\int_{y_2}^{y_2+y_3}n_i(\alpha_1,\alpha_2)d\alpha_1d\alpha_2,\quad\text{for $i=1,2,3$}.
\end{equation}
Then the map $\mathscr T_{\tilde x}=\mathcal{M}_y^{-1}:x\rightarrow y$ makes the boundary flat and satisfies
\begin{equation*}
\mathscr T_{\tilde x}(U_{\delta}\cap{\Omega})\rightarrow B^+_{R}\quad\text{and}\quad \mathscr T_{\tilde x}(\partial U_{\delta}\cap{\Omega})\rightarrow \partial B^+_{R}\cap\{y_3=0\},
\end{equation*}
where $U_{\delta}$ is a neighborhood of $\tilde x$, and  $B_R^+=\{y_1^2+y_2^2+y_3^2<R,\ y_3>0\}$ with $\delta$ and $R$ uniform constants along the boundary of $\partial\Omega$. Denote the jacobian $\big(\frac{\partial y_i}{\partial x_j}\big)=D(x)$, then for any $\xi\in\mathbb{R}^3$, there exists a constant $C$ such that
\begin{equation}\label{jacobbi}
C^{-1}|\xi|\leq|D(x)\xi|\leq C|\xi|\quad\text{and}\quad C^{-1}|\xi|\leq|D^{-1}(x)\xi|\leq C|\xi|.
\end{equation}Moreover, the map also satisfies \text{ for $x\in\partial\Omega \  (i.e. \ y_3=0)$},
\begin{equation}\label{mapproperty}
\sum_{i=1}^3\frac{\partial y_j}{\partial x_i}\frac{\partial y_3}{\partial x_i}=0, \text{ for }j=1,2\text{ and } \bigg(\frac{\partial y_3}{\partial x_1},\frac{\partial y_3}{\partial x_2},\frac{\partial y_3}{\partial x_3}\bigg)\times {\textbf{n}}=0.\end{equation}
On the boundary $\partial\Omega$, denote
\begin{equation*}
g_1=\sum_{i,j=1}^3a_{ij}\partial_{j}\mathcal U n_i,\quad g_2=\sum_{i=1}^3b_i n_i \quad\text{and}\quad g_3=-\bar q n_3.
\end{equation*}
It follows from \eqref{bi} and \eqref{aijni} that
\begin{equation}\label{x g}
|g_1|\leq\frac{C}{\tilde x_3^{\mathfrak b}},\quad|g_2|\leq\frac{C}{\tilde x_3^{\mathfrak b_1}}\quad\text{and}\quad |g_3|\leq\frac{C}{\tilde x_3^{\mathfrak a_1+1}}\quad\text{on $\partial\Omega$},
\end{equation}
where $\mathfrak b=\min(\mathfrak a_1,\mathfrak b_1)$.

In the rest of the paper, $\tilde f$ denotes the function $f$ in $y$-coordinates. Because of \eqref{mapproperty}, one has
\begin{equation}\label{tilde g}
\sum_{i,j,l=1}^3\tilde a_{ij}\frac{\partial \mathcal U}{\partial y_l}\frac{\partial y_l}{\partial x_j}\frac{\partial y_3}{\partial x_i}\frac{1}{\mathcal W}=  \tilde g_1 \quad\text{and}\quad\sum_{i,l=1}^3\frac{\partial \mathcal U}{\partial y_l}\frac{\partial y_l}{\partial x_i}\frac{\partial y_3}{\partial x_i}\frac{1}{\mathcal W}=\frac{\partial \mathcal U}{\partial y_3}=\tilde g_3,
\end{equation} where $\mathcal W=\bigg(\sum\limits_{i=1}^3|\frac{\partial y_3}{\partial x_i}|^2\bigg)^{\frac{1}{2}}$.  It follows from \eqref{x g} that
\begin{equation}
|\tilde g_1|\leq\frac{C}{\tilde x_3^{\mathfrak b}},\quad|\tilde g_2|\leq\frac{C}{\tilde x_3^{\mathfrak b_1}}\quad\text{and}\quad|\tilde g_3|\leq\frac{C}{\tilde x_3^{\mathfrak a_1+1}}\quad \text{on $\partial B^+_{R}\cap\{y_3=0\}$}.
\end{equation}
Denote $g_4=\sum_{i=1}^3\partial_ib_i$, the straightforward computations give
\begin{equation}
 g_4\leq C\tilde x_3^{-\mathfrak b_1} \quad \text{and}\quad \tilde g_4\leq C\tilde x_3^{-\mathfrak b_1}.\end{equation}

After changing variables, the problem \eqref{mathcalU} can be written as follows
\begin{equation}\label{Y_equation}
\begin{cases}
\sum\limits_{i,j,l,s=1}^3\frac{\partial}{\partial y_s}\bigg(\tilde a_{ij}(y)\frac{\partial \mathcal U}{\partial y_l}\frac{\partial y_l}{\partial x_j}\bigg)\frac{\partial y_s}{\partial x_i}+\tilde g_4=0,&\quad\text{ in} B_R^+,\\[3mm]
\sum\limits_{i,s=1}^3\frac{\partial\mathcal U}{\partial y_s}\frac{\partial y_s}{\partial x_i}\frac{\partial y_3}{\partial x_i}=\tilde g_3\mathcal W, &\quad\text{ on }  B_R^+\cap\{y_3=0\}.
\end{cases}
\end{equation}
For any $\psi\in C^3_0(B_R^+)$, multiplying $\psi$ on both sides of \eqref{Y_equation} and integarting by parts yield
\begin{equation}
\sum_{i,j,s,l=1}^3\int_{B_{R}^+}\tilde a_{ij}\frac{\partial \mathcal U}{\partial y_l}\frac{\partial y_l}{\partial x_j}\frac{\partial y_s}{\partial x_i}\frac{\partial\psi}{\partial y_s}dy-\int_{B_R^+}\tilde g_4\psi dy=\sum_{i,j,l=1}^3\int_{ B_{R}^+\cap\{y_3=0\}}\tilde a_{ij}\frac{\partial \mathcal U}{\partial y_l}\frac{\partial y_l}{\partial x_j}\frac{\partial y_3}{\partial x_i}\psi dy'.
\end{equation}
Denote $A_{sl}=\sum\limits_{i,j=1}^3\tilde a_{ij}\frac{\partial y_s}{\partial x_j}\frac{\partial y_l}{\partial x_i}$. In virtue of \eqref{jacobbi}, we still using $\lambda$ and $\Lambda$ such that \begin{equation}\label{partialA}
\lambda|\xi|^2\leq\sum_{s,l=1}^3 A_{sl}\xi_s\xi_l\leq\Lambda|\xi|^2\quad \text{and}\quad \bigg|\frac{\partial A_{ls}}{\partial y_r}\bigg|\leq C.
\end{equation}It follows from \eqref{tilde g} that
\begin{equation}\label{Psi}
\sum_{s,l=1}^3\int_{B^+_R}A_{ls}\frac{\partial\mathcal U}{\partial y_l}\frac{\partial\psi}{\partial y_s}dy-\int_{B_R^+}\tilde g_4\psi dy=\int_{ B_R^+\cap\{y_3=0\}}\psi\tilde g_1\mathcal W dy_1dy_2.
\end{equation}
Denote $g=\frac{\tilde g_1\mathcal W}{A_{33}}$. By the definition of $A_{33}$, one has
\begin{equation}
\lambda \mathcal W^2\leq \tilde a_{ij}\frac{\partial y_3}{\partial x_i}\frac{\partial y_3}{\partial x_j}=A_{33}\leq \Lambda \mathcal W^2.
\end{equation}

Given $\varsigma(z')\in C^2_0(\mathbb R^2)$ satisfying $\int_{\mathbb R^2}\varsigma(z') dz'=1$, define
\begin{equation}\label{vartheta1}
\vartheta(y)=y_3\int _{\mathbb R^2}g(y'-y_3z')\varsigma(z')dz'.
\end{equation}
Then
\begin{equation}\label{varthetaboundary}
\vartheta (y',0)=\frac{\partial\vartheta}{\partial y_1}(y',0)=\frac{\partial\vartheta}{\partial y_2}(y',0)=0 \quad \text{and}\quad \frac{\partial\vartheta}{\partial y_3}(y',0)=g(y').
\end{equation}
The straightforward computations yield
\begin{equation}\label{vartheta}
\|\vartheta\|_{C^2{(B_R^+)}}\leq \frac{C}{\tilde x_3^{\mathfrak b}}.
\end{equation}
Define
$\mathfrak u=\partial_s( A_{sl}\partial_l\vartheta)$ and $\kappa=\mathfrak u +\tilde g_4$. It follows from \eqref{vartheta} and the definition of $A_{sl}$ that one has
\begin{equation}\label{f}
\|\mathfrak u\|_{L^{\infty}( B_{R}^+)}\leq \frac{C}{\tilde x_3^{\mathfrak b}},\quad \|\kappa\|_{L^{\infty}( B_{R}^+)}\leq \frac{C}{\tilde x_3^{\mathfrak b}}
\end{equation}
and \text{for any }$\psi\in C_0^3(B^+_R)$
\begin{equation}\label{vartheta_estimate}-\sum_{s,l=1}^3\int_{B_R^+}A_{ls}\partial_l\vartheta\partial_s\psi+\sum_{l=1}^3\int_{ B_{R}^+\cap\{y_3=0\}}A_{l3}\partial_l \vartheta\psi dx=\int_{B_R^+}(\mathfrak u+\tilde g_3)\psi dx.\end{equation}
Combining \eqref{Psi} and \eqref{vartheta_estimate} yields
\begin{equation}\label{u}
\sum_{s,l=1}^3\int_{B^+_R}A_{ls}\frac{\partial(\mathcal U-\vartheta)}{\partial y_l}\frac{\partial\psi}{\partial y_s}dy=\int_{B_R^+}\kappa\psi dy,
\end{equation} where the boundary conditions \eqref{varthetaboundary} have been used.
Denote $v=\mathcal U-\vartheta$. Replacing $\psi$ by each $\frac{\partial\psi}{\partial y_i}\ (i=1,2,3)$ in \eqref{u} and integrating by parts yield
\begin{equation}\label{u_estimate}
\begin{split}
\quad&\int_{ B_R^+}\kappa\frac{\partial \psi}{\partial y_i}dy=-\sum_{l,s=1}^3\int_{B_R^+}A_{ls}\frac{\partial}{\partial y_l}\bigg(\frac{\partial v}{\partial y_i}\bigg)\frac{\partial\psi}{\partial y_s}dy\\&-\sum_{l,s=1}^3\int_{B_R^+}\frac{\partial A_{ls}}{\partial y_i}\frac{\partial v}{\partial y_l}\frac{\partial\psi}{\partial y_s}dy+\delta_{i3}\sum_{l,s=1}^3\int_{ B_{R}^+\cap\{y_3=0\}}A_{ls}\frac{\partial v}{\partial y_l}\frac{\partial\psi}{\partial y_s}ds,\quad\text{for $i=1,2,3$}.
\end{split}
\end{equation}

Define
\begin{equation*}
\Theta=\max\limits_{ B_R^+\cap\{y_3=0\}}\bigg|\frac{\partial v}{\partial y_3}\bigg|,\quad w_1=\frac{\partial v}{\partial y_1},\quad w_2=\frac{\partial v}{\partial y_2}\quad\text{and}\quad w_3=\frac{\partial v}{\partial y_3}-\Theta.\end{equation*}
It follows from \eqref{tilde g} and \eqref{vartheta} that  \begin{equation}\label{Theta}
\Theta\leq\max\limits_{ B_R^+\cap\{y_3=0\}}\bigg|\frac{\partial \mathcal U}{\partial y_3}\bigg|+\max\limits_{ B_R^+\cap\{y_3=0\}}\bigg|\frac{\partial \vartheta}{\partial y_3}\bigg|\leq \frac{C}{\tilde x_3^{\mathfrak b}}.\end{equation}
Moreover, the expression \eqref{u_estimate} can be written as, for $i=(1,2,3)$
\begin{equation}
\label{uw_estimate}
\begin{split}
&\quad\sum_{l,s=1}^3\int_{B_R^+}A_{ls}\frac{\partial w_i}{\partial y_l}\frac{\partial\psi}{\partial y_s}dy+\sum_{l,s=1}^3\int_{B_R^+}w_l\frac{\partial A_{ls}}{\partial y_i}\frac{\partial \psi}{\partial y_s}dy\\[3mm]
&=\delta_{i3}\sum_{l,s=1}^3\int_{ B_{R}^+\cap\{y_3=0\}}A_{ls}\frac{\partial v}{\partial y_l}\frac{\partial\psi}{\partial y_s}ds-\int_{ B_R^+}\kappa\frac{\partial \psi}{\partial y_i}dy-\delta_{l3}\Theta\sum_{l,s=1}^3\int_{B_R^+}\frac{\partial A_{ls}}{\partial y_i}\frac{\partial \psi}{\partial y_s}dy.
\end{split}
\end{equation}

Now we use Nash-Moser iteration to get the $L^\infty-$norm of $w_i$. We consider only the case $w_i\geq 0$. If  $w_i\geq0$ does not hold, one can repeat the proof for $w_i^+$ and $w_i^-$, respectively. It is easy to see that
\begin{equation}\label{w3}
w_3=0 \quad\text{on $B_{R}^+\cap\{y_3=0\}$}.\end{equation}
For $i=1,2,3$, denote $\psi_i=\eta^2{w}_i^{\mu+1}$ with some $\mu\geq0$ and some nonnegative function $\eta\in C^2_0({B^+_R})$.
Direct calculations give
\begin{equation*}
\frac{\partial\psi_i}{\partial y_k}=2\eta\frac{\partial \eta}{\partial y_k}{w}_i^{\mu+1}+\eta^2(\mu+1)\frac{\partial w_i}{\partial y_k}{w_i}^\mu, \quad\text{for $k=1,2,3$}.
\end{equation*}
If one replaces $\psi$ by $\psi_i$ $(i=1,2,3)$ in \eqref{uw_estimate}, then it holds
\begin{equation}
\begin{split}
&\quad\sum_{l,s=1}^3\int_{B_R^+}A_{ls}\frac{\partial w_i}{\partial y_l}\frac{\partial\psi_i}{\partial y_s}dy+\sum_{l,s=1}^3\int_{B_R^+}w_l\frac{\partial A_{ls}}{\partial y_i}\frac{\partial \psi_i}{\partial y_s}dy\\
&=-\int_{ B_R^+}\kappa\frac{\partial \psi_i}{\partial y_i}dy-\Theta\sum_{s=1}^3\int_{B_R^+}\frac{\partial A_{3s}}{\partial y_i}\frac{\partial \psi_i}{\partial y_s}dy\\[2mm]
&=-\sum_{s=1}^3\int_{B_R^+}(\delta_{is}\kappa+\Theta\frac{\partial A_{3s}}{\partial y_i})\frac{\partial\psi_i}{\partial y_s}dy, \quad\text{for $i=1,2,3$}.\end{split}
\end{equation}
where the boundary term vanishes due to \eqref{w3}.
For $i=1,2,3$, the straightforward computations give
\begin{equation}\label{A_estimate}
\begin{split}
&\quad\sum_{l,s=1}^3\int_{B_R^+}A_{ls}\frac{\partial w_i}{\partial y_l}\bigg(2\eta\frac{\partial \eta}{\partial y_s}{w}_i^{\mu+1}+\eta^2(\mu+1)\frac{\partial w_i}{\partial y_s}{w_i}^\mu\bigg) dy\\[2mm]
&\geq\lambda(\mu+1)\int_{B_R^+}\eta^2w_i^\mu|Dw_i|^2
dy-2\Lambda\int_{B_R^+}\eta w_i^{\mu+1}|D\eta||Dw_i|dy\\[2mm]&\geq\lambda(\mu+1)\int_{B_R^+}\eta^2w_i^\mu|Dw_i|^2
dy-\epsilon\int_{B_R^+}\eta^2w_i^\mu|Dw_i|^2-\frac{1}{\epsilon}\int_{B_R^+}|D\eta|^2w_i^{\mu+2}dy
\end{split}
\end{equation}
and
\begin{equation}\label{w_A_estimate}
\begin{split}
&\quad\sum_{l,s=1}^3\int_{B_R^+}w_l\frac{\partial A_{ls}}{\partial y_i}\frac{\partial \psi_i}{\partial y_s}dy\\&\leq C\sum_{l,s=1}^3\int_{B_R^+} \eta|D\eta|w_i^{\mu+1}\bigg|\frac{\partial A_{ls}}{\partial y_i}w_l\bigg|+\bigg|\frac{\partial A_{ls}}{\partial y_i}w_l\bigg|\eta^2(\mu+1)w_i^{\mu}|Dw_i|dy\\
&\leq C\sum_{l,s=1}^3\int_{B_R^+}|D\eta|^2w_i^{\mu+2}+\bigg(\eta^2+\frac{1}{\epsilon}\bigg)\bigg|\frac{\partial A_{ls}}{\partial y_i}w_l\bigg|^2w_i^\mu+\epsilon\eta^2(\mu+1)^2w_i^\mu|Dw_i|^2dy\\
&\leq C\int_{B_R^+}|D\eta|^2w_i^{\mu+2}+\eta^2|\bar w|^2w_i^\mu+\epsilon\eta^2(\mu+1)^2w_i^\mu|Dw_i|^2+\frac{1}{\epsilon}|\bar w|^2\eta^2w_i^\mu dy,
\end{split}
\end{equation}
where  $|\bar w|^2=w_1^2+w_2^2+w_3^2$.
Denote
\begin{equation}\mathcal F_{is}=\delta_{is}\kappa+\Theta\frac{\partial A_{3s}}{\partial y_i} \quad\text{and} \quad\mathcal K=\max|\mathcal F_{is}|.\end{equation} It follows from \eqref{f} and \eqref{Theta} that  \begin{equation}\label{F}
|\mathcal K|\leq \frac{C}{\tilde x_3^{\mathfrak b}}.
\end{equation}
If $|w_i|\leq \mathcal K$, the estimate \eqref{Thm velocity} holds. Hence we assume
\begin{equation}\label{wi}|w_i|\geq \mathcal K,\quad (i=1,2,3).\end{equation} Therefore, one has
\begin{equation}\label{F_estimate}\begin{split}
&\quad\sum_{s=1}^3\int_{B_{R}^+}\mathcal F_{is}\frac{\partial\psi_i}{\partial y_s}dy\leq \int_{B_R^+}\mathcal K\eta|D\eta|w_i^{\mu+1}+\mathcal K\eta^2(\mu+1)w_i^\mu|Dw_i|dy\\
&\leq  \int_{B_R^+}\eta|D\eta|w_i^{\mu+2}+\epsilon\eta^2(\mu+1)^2w_i^\mu|Dw_i|^2+\frac{1}{\epsilon}\eta^2\mathcal K^2w_i^\mu dy\\
&\leq \int_{B_R^+}\eta|D\eta|w_i^{\mu+2}+\epsilon\eta^2(\mu+1)^2w_i^\mu|Dw_i|^2+\frac{1}{\epsilon}\eta^2w_i^{\mu+2} dy.  \end{split}
\end{equation}
Combining \eqref{A_estimate}, \eqref{w_A_estimate}, and \eqref{F_estimate} yields
\begin{equation}
\begin{split}
&\quad\lambda(\mu+1)\int_{B_R^+}\eta^2w_i^\mu|Dw_i|^2dy- \big(\epsilon+2\epsilon(\mu+1)^2\big)
\int_{ B_R^+}\eta^2w_i^\mu|Dw_i|^2dy\\
&\leq (C+\frac{C}{\epsilon})\int_{ B_R^+}\eta^2|\bar w|^2w_i^\mu dy+(C+\frac{C}{\epsilon})\int_{ B_R^+}(\eta^2+|D\eta|^2)w_i^{\mu+2} dy.
\end{split}\end{equation}
If one chooses $\epsilon=\frac{\lambda}{8(\mu+1)}$, then it holds that
\begin{equation}
\int_{ B_R^+}\eta^2w_i^2|Dw_i|^2dy\leq C\int_{ B_R^+}|D\eta|^2|w_i|^{\mu+2}+\eta^2|w_i|^{\mu+2}+\eta^2|\bar w|^2|w_i|^\mu dy.
\end{equation}
Therefore, one has
\begin{equation}
\int_{ B_R^+}\bigg|D(\eta w_i^\frac{\mu+2}{2})\bigg|^2dy\leq C(\mu+2)^2\int_{ B_R^+}|D\eta|^2|w_i|^{\mu+2}+\eta^2|w_i|^{\mu+2}+\eta^2|\bar w|^2|w_i|^\mu dy.
\end{equation}
Applying Sobolev inequality yields
\begin{equation}\label{Sobolev}
\bigg(\int_{ B_R^+}(\eta w_i^\frac{\mu+2}{2})^6dy\bigg)^\frac{1}{3}\leq C(\mu+2)^2\int_{ B_R^+}|D\eta|^2|w_i|^{\mu+2}+\eta^2|w_i|^{\mu+2}+\eta^2|\bar w|^2|w_i|^\mu dy.
\end{equation}

Set \begin{equation*}
R_j=(\frac{1}{2}+\frac{1}{2^{j+1}})R\quad\text{and}\quad \gamma_j=2\cdot 3^j.
\end{equation*}
Let $\eta_j\in C_0^\infty(B_{R_j}^+)$ satisfy \begin{equation*}
\quad \eta_j=1\text{ in } B_{R_{j+1}}^+ \quad\text{and}\quad |D\eta_j|\leq\frac{4}{R_j-R_{j+1}}.
\end{equation*}
Choose $\mu=\gamma_j-2$, it follows from \eqref{Sobolev} that
\begin{equation*}
\bigg(\int_{ B_{R_{j+1}}^+}w_i^{\gamma_{j+1}}dy\bigg)^\frac{1}{3}\leq C\gamma_j^2\int_{ B_{R_j}^+}\big(2^{j+1}R\big)^2w_i^{\gamma_j}+w_i^{\gamma_j}+|\bar w|^2w_i^{\gamma_j-2} dy.
\end{equation*}
Thus,
\begin{equation}
\bigg(\int_{ B_{R_{j+1}}^+}w_i^{\gamma_{j+1}}dy\bigg)^{\frac{1}{\gamma_{j+1}}}\leq\bigg(\int_{ B_{R_j}^+}A_jw_i^{\gamma_j}+B_jw_i^{\gamma_j}+B_j|\bar w|^2w_i^{\gamma_j-2} dy\bigg)^\frac{1}{\gamma_j},
\end{equation}
where $A_j=C\gamma_j^2\big(2^{j+1}/R\big)^2$ and $B_j=C\gamma_j^2$.
Note that
\begin{equation}
\int_{B_{R_j}^+}|\bar w|^2w_i^{\gamma_j-2}dy\leq\bigg(\int_{ B_{R_j}^+}w_i^{\gamma_j}\bigg)^\frac{\gamma_j-2}{\gamma_j}\bigg(\int_{ B_{R_j}^+}|\bar w|^{\gamma_j}\bigg)^{\frac{2}{\gamma_j}}.
\end{equation}
Therefore, one has
\begin{equation}
\begin{split}
&\quad\bigg(\int_{ B_{R_{j+1}}^+}w_i^{\gamma_{j+1}}dy\bigg)^{\frac{1}{\gamma_{j+1}}}\\
&\leq \bigg[A_j\int_{ B_{R_j}^+}w_i^{\gamma_j}+B_j\int_{ B_{R_j}^+}w_i^{\gamma_j}+B_j\bigg(\int_{ B_{R_j}^+}w_i^{\gamma_j}\bigg)^\frac{\gamma_j-2}{\gamma_j}\bigg(\int_{ B_{R_j}^+}\bar w^{\gamma_j}\bigg)^{\frac{2}{\gamma_j}}\bigg]^\frac{1}{\gamma_j}\\[2mm]
&\leq \|w_i\|_{L^{\gamma_j}(B_{R_{j}}^+)}^{\frac{\gamma_j-2}{\gamma_j}}\bigg[(A_j+B_j)\|w_i\|_{L^{\gamma_j}(B_{R_{j}}^+)}^2+B_j\|\bar w\|_{L^{\gamma_j}(B_{R_{j}}^+)}^2\bigg]^{\frac{1}{\gamma_j}}\\[2mm]&\leq\|w_i\|_{L^{\gamma_j}(B_{R_{j}}^+)}^{\frac{\gamma_j-2}{\gamma_j}}(A_j+2B_j)^{\frac{1}{\gamma_j}}\|\bar w\|_{L^{\gamma_j}(B_{R_{j}}^+)}^\frac{2}{\gamma_j}.
\end{split}
\end{equation}
This implies that
\begin{equation}\label{w}
\begin{split}
&\quad\sum\limits_{i=1}^{3}\|w_i\|_{L^{\gamma_j+1}(B_{R_{j+1}}^+)}\\&\leq\sum\limits_{i=1}^3\|w_i\|_{L^{\gamma_j}(B_{R_{j}}^+)}^{\frac{\gamma_j-2}{\gamma_j}}(A_j+2B_j)^{\frac{1}{\gamma_j}}\|\bar w\|_{L^{\gamma_j}(B_{R_{j}}^+)}^\frac{2}{\gamma_j}\\
&\leq \bigg[\sum\limits_{i=1}^3\bigg(\|w_i\|_{L^{\gamma_j}(B_{R_{j}}^+)}^{\frac{\gamma_j-2}{\gamma_j}}(A_j+2B_j)^{\frac{1}{\gamma_j}}\|\bar w\|_{L^{\gamma_j}(B_{R_{j}}^+)}^\frac{2}{\gamma_j}\bigg)^{\frac{\gamma_j}{\gamma_j-2}}\bigg]^\frac{\gamma_j-2}{\gamma_j}3^{\frac{2}{\gamma_j}}\\
&\leq(9A_j+18B_j)^\frac{1}{\gamma_j}\bigg(\sum\limits_{i=1}^{3}\|w_i\|_{L^{\gamma_j}(B_{R_{j}}^+)}\bigg)^\frac{2}{\gamma_j}\bigg(\sum\limits_{i=1}^{3}\|w_i\|_{L^{\gamma_j}(B_{R_{j}}^+)}\bigg)^{\frac{\gamma_j-2}{\gamma_j}}\\
&\leq (9A_j+18B_j)^\frac{1}{\gamma_j}\sum\limits_{i=1}^{3}\|w_i\|_{L^{\gamma_j}(B_{R_{j}}^+)}.
\end{split}
\end{equation}

Set \begin{equation*}Q_{j+1}=\sum\limits_{i=1}^{3}\|w_i\|_{L^{\gamma_j+1}(B_{R_{j+1}}^+)}\quad \text{and} \quad S_j=(9A_j+18B_j)^\frac{1}{\gamma_j}.
\end{equation*}
Then the estimate \eqref{w} can be written as
\begin{equation}
Q_{j+1}\leq S_j Q_j.
\end{equation}
Obviously,
\begin{equation}
S_j=(9A_j+18B_j)^\frac{1}{\gamma_j}\leq \bigg(C\gamma_j^2\big(2^{j+1}/R\big)^2\bigg)^\frac{1}{\gamma_j}\leq C^\frac{1}{\gamma_j}16^{\frac{j}{\gamma_j}}.
\end{equation}
Hence,
\begin{equation*}
Q_{j+1}\leq C^{\sum\limits_{i=1}^j\frac{1}{\gamma_i}}16^{\sum\limits_{i=1}^j\frac{i}{\gamma_i}}Q_0.
\end{equation*}
Note that
\begin{equation*}
\sum\limits_{i=1}^j\frac{1}{\gamma_i}\leq C\quad \text{and}\quad\sum\limits_{i=1}^j\frac{i}{\gamma_i}\leq C.
\end{equation*}
Taking $j\rightarrow\infty$ yields
\begin{equation}
\sum\limits_{i=1}^3\|w_i\|_{L^\infty(B_{\frac{1}{2}R}^+)}\leq C\sum\limits_{i=1}^3\|w_i\|_{L^2(B_R^+)},
\end{equation}
provided that \eqref{wi} holds.
Therefore, one has
\begin{equation}\label{nash-moser}
\sum\limits_{i=1}^3\|w_i\|_{L^\infty(B_{\frac{1}{2}R}^+)}\leq C\sum\limits_{i=1}^3\|w_i\|_{L^2(B_R^+)}+\mathcal K.
\end{equation}
It follows from the definition of $w_i$, \eqref{vartheta}, \eqref{f} and \eqref{Theta} that
\begin{equation}\label{bdygradientPsi}
\|\nabla\mathcal U\|_{L^{\infty}(U_\delta)}\leq C\big(\|\nabla\mathcal U\|_{L^2(U_\delta)}+\mathcal K+\Theta+\|\vartheta\|_{C^2{(B_R^+)}}\big).
\end{equation}
As same as the estimate for \eqref{bdygradientPsi} with $\vartheta=0$, $\kappa=\mathfrak u$ and $\Theta=0$, for any $B_R\in\Omega$, one has
\begin{equation}
\|\nabla\mathcal U\|_{L^\infty(B_{\frac{R}{2}})}\leq C\|\nabla\mathcal U\|_{L^2(B_R)}.
\end{equation}
In a word, when the boundary satisfies the convergence rate \eqref{Iboundary_algebratic_rateii} and the external force $\phi$ satisfies \eqref{force algebratic rate}, we have
\begin{equation}\label{step2 nash-moser}
|\nabla\varphi^\epsilon-(0,0,\bar q))|\leq Cx_3^{-\mathfrak b}.
\end{equation}

For the case where the nozzle is a perfect cylinder if $x_3$ is  sufficiently large, the same as \eqref{step2 nash-moser} with $g_1=g_2=g_3=g_4=0$, we can prove that
there exists a positive constant $\mathfrak d$ such that
\begin{equation}\label{cylinder boundary'}
|\nabla\varphi^\epsilon-\nabla\varphi_*|\leq Ce^{-\mathfrak d x_3}.
\end{equation}

For the case where the nozzle boundary satisfies \eqref{Iboundary_algebratic_rateii} with $\mathfrak a_1>1$, one can follow the proof of estimate \eqref{step2 nash-moser} with $g_2=g_4=0$, $|g_1|\leq CT^{-\mathfrak a_1}$, and $|g_3|\leq CT^{-\mathfrak a_1}$ to show
\begin{equation}\label{cylinder boundary''}
\|\nabla\varphi^\epsilon-\nabla\varphi_*\|_{L^\infty(\Omega(T,T+1))}\leq C\|\nabla W\|_{L^2(\Omega(T,T+1))}+ CT^{-\mathfrak a_1}\leq  C T^{-\mathfrak a_1+1}.
\end{equation}
Hence the proof of Theorem \ref{Theorem convergence rate} is completed.
Furthermore, one can also use the Nash-Moser iteration to get the desired estimates \eqref{Icylinder boundary} and \eqref{Ibdecay} in Theorem \ref{incompressible case}.

\medskip
	\textbf{Acknowledgement.}
	The research of Lei Ma and Chunjing Xie were partially supported by NSFC grants 11971307, 11631008, and 11422105. Xie is also supported by Young Changjiang Scholar of Ministry of Education in China. The research of Tian-Yi Wang was supported in part by the NSFC Grant No. 11601401 and 11971024.
	\nocite{*}
	\bibliographystyle{abbrv}
	\bibliography{ref}

 \end{document}